\documentclass[a4paper, 12pt, reqno]{amsart}

\usepackage{fancyhdr}
\usepackage{amssymb,amsmath}
\usepackage[dvips]{graphics}
\usepackage[dvips]{color}
\usepackage{graphicx}
\usepackage[english]{babel}
\usepackage[utf8]{inputenc}
\usepackage{comment}
\usepackage{todonotes}
\usepackage{esint}

\newtheorem{theorem}{Theorem}[section]
\newtheorem{lemma}[theorem]{Lemma}

\theoremstyle{definition}
\newtheorem{definition}[theorem]{Definition}

\newtheorem{remark}[theorem]{Remark}

\numberwithin{equation}{section}

\title[A Harnack estimate for supersolutions]{A weak Harnack estimate for supersolutions to the porous medium equation}

\author[Pekka Lehtel\"a]{Pekka Lehtel\"a}

\thanks{The research is supported by the Emil Aaltonen Foundation.}

\newcommand\rn{\mathbb R^n}
\newcommand\re{\mathbb R}
\newcommand\n{\mathbb N}

\newcommand\bd{\partial}
\newcommand\ph{\varphi}
\newcommand\eps{\varepsilon}

\newcommand\supp{\operatorname{supp}}
\newcommand\sgn{\operatorname{sign}}
\DeclareMathOperator*{\esssup}{ess\,sup}
\DeclareMathOperator*{\essinf}{ess\,inf}
\newcommand\dx {\, d}

\providecommand{\ch}[1]{\text{\raise 2pt \hbox{$\chi$}\kern-0.2pt}_{#1}}

\providecommand{\vint}[1]{\mathchoice
          {\mathop{\vrule width 5pt height 3 pt depth -2.5pt
                  \kern -9pt \kern 1pt\intop}\nolimits_{\kern -5pt{#1}}}%
          {\mathop{\vrule width 5pt height 3 pt depth -2.6pt
                  \kern -6pt \intop}\nolimits_{\kern -3pt{#1}}}%
          {\mathop{\vrule width 5pt height 3 pt depth -2.6pt
                  \kern -6pt \intop}\nolimits_{\kern -3pt{#1}}}%
          {\mathop{\vrule width 5pt height 3 pt depth -2.6pt
                  \kern -6pt \intop}\nolimits_{\kern -3pt{#1}}}}
                  
\begin{document}

\begin{abstract}
In this work, we prove a weak Harnack estimate for the weak supersolutions to the porous medium equation. The proof is based on a priori estimates for the supersolutions and measure theoretical arguments. 
\end{abstract}



\date{\today}  
\maketitle

\section{Introduction}
In this paper, we will prove a weak Harnack estimate for the weak supersolutions to the porous medium equation
\begin{equation}
  \label{eq:PME}
  u_t - \Delta u^m =0 \quad\text{in } \Omega_{T_0},
\end{equation}
where $\Omega_{T_0}=\Omega\times (0,T_0)$ and $\Omega\subset \rn$ is a bounded domain. In this work, we consider only the degenerate case $m>1$. In the singular case $m<1$, a different proof is needed. For the general theory of the equation, we refer to \cite{daskalopoulos}, \cite{vazquez} and \cite{nonlinear diffusion}.
 
We consider the weak supersolutions, which are defined in the usual way, with test functions under the integral sign, as weak solutions to the inequality 
\[
u_t- \Delta u^m \ge 0.
\]
Throughout the work, we assume, that the weak supersolutions are non-negative. Properties of supersolutions to the PME are considered in \cite{supersolutions to PME} and \cite{unbounded supersolutions}. In the latter one, also unbounded supersolutions, which are defined via comparison to weak solutions, are treated. In the case of the evolutionary $p$-Laplace equation, some interesting phenomena are observed in the case of unbounded supersolutions, see \cite{shadows}. However, in the theory of the porous medium equation, there is a missing link, namely the weak Harnack estimate for the supersolutions, which is the main result of this work.
\begin{theorem}\label{main theorem}
Let $u>0$ be a weak supersolution in $\Omega_{T_0}\supset B(x_0,8\rho)\times(0,T_0)$. Then, there exist constants $C_1,C_2>0$ depending on $m$ and $n$, such that for almost every $t_0\in(0,T_0)$, the following inequality holds
\[
\fint_{B(x_0,\rho)} u(x,t_0)\dx x \le   \left( \frac{C_1 \rho^2}{T_0-t_0}\right ) ^{1/(m-1)} + C_2 \essinf_Q u,
\]
where 
\begin{align*}
&Q=B(x_0,4\rho) \times (t_0+\tau/2, t_0+\tau)\quad \text{and}\\
&\tau=\min \left\{ T_0-t_0, C_1\rho^2 \left ( \fint_{B(x_0,\rho)} u(x,t_0) \dx x \right )^{-(m-1)}\right\}.
\end{align*}
\end{theorem}
Apart from the classification of unbounded supersolutions, the Harnack estimates play a significant role in the regularity theory of partial differential equations. The parabolic Harnack estimates have attracted a lot of interest since the result of Moser in \cite{moser}. More recently, the Harnack's inequality for the weak solutions to $p$-Laplace type equations was proved in \cite{harnack1}. Later, the weak Harnack estimate for weak supersolutions to $p$-Laplace type equations was proved in \cite{kuusi}. Finally, the various Harnack estimates for nonlinear parabolic equations are collected in \cite{harnack2}, where also the result for weak supersolutions to the PME is presented. 

Even though the result is probably known to experts, it seems to be difficult to find a reference with a complete proof. The purpose of this work is to present a proof for the weak Harnack estimate in full detail. As pointed out in \cite{harnack2}, the structure of the proof is similar to the one of $p$-Laplace type equations. However, there are numerous issues to be taken care of. For instance, constants cannot be added to solutions. Moreover, in the energy estimates, usually we can only control the norm of $\nabla u^m$, which raises some technical challenges in the arguments. 

One novelty of our proof is that we are able to bypass the assumption $u>\delta>0$ in the Caccioppoli estimates and assume only $u>0$. This is done by choosing a clever test function, introduced in the context of doubly nonlinear equation in \cite{ivert}. Thus we are able to prove the Harnack estimate for positive supersolutions directly, without approximation by supersolutions, that are bounded away from zero.

\section{Preliminaries}
Throughout the work, we will denote a bounded domain in $\rn$ by $\Omega$. Moreover, we make a technical assumption $B(x_0,8\rho)\subset \Omega$. We work with the space-time cylinders $\Omega_{T_0}=\Omega\times (0,T_0) \subset \re^{n+1}$. The parabolic boundary $\bd_p U$ of a space-time cylinder $U=B\times (t_1,t_2)$ is defined as $\bd_p U = B\times \{t_1\} \cup \bd B \times (t_1,t_2)$. Similarly, the conjugate parabolic boundary is defined as $\bd^p U = B\times \{t_2\} \cup \bd B \times (t_1,t_2)$.

\begin{definition}
  A function $u\in L^2_{loc}(0,T; H^1_{loc}(\Omega))$ is a weak supersolution to \eqref{eq:PME}, if $u^m\in L^2_{loc}(0,T;H^1_{loc}(\Omega))$ and $u$ satisfies
\[
\iint_{\Omega_T}\big(-u \ph_t + \nabla u^m \cdot \nabla \ph \big) \dx x \dx t \ge 0
\]
for all test functions $\ph\in C_0^\infty(\Omega_T)$, such that $\ph\ge 0$. 
\end{definition}

First, we will prove some Caccioppoli type estimates for the supersolutions. In order to prove the estimates, we need to use a test function depending on the supersolution $u$ itself. However, no regularity for $u$ is assumed in the time variable, and thus we need to regularize the function. We will use the averaged function
\begin{equation}\label{averaged function}
u^*(x,t)= \frac 1 \sigma \int_0^t e^{\frac{s-t}\sigma} u(x,s) \dx s 
\end{equation}
to avoid the possibly nonexistent quantity $u_t$. Now, the function $u^*$ satisfies the following inequality
\begin{equation}\label{averaged eq}
\int_{\tau_1}^{\tau_2} \int_\Omega \left(  (\nabla u^m)^* \cdot \nabla \ph + \ph \frac{\partial u^*}{\partial t}\right)  \dx x \dx t \ge 0
\end{equation}
for every non-negative test function $\ph\in L^2(0,T;H_0^1(\Omega))$. For the properties of $u^*$, we refer to \cite{supersolutions to PME}. 

\begin{lemma}\label{caccioppoli}
  Let $u$ be a non-negative weak supersolution in $\Omega_T$ and let $\zeta\in C_0^\infty(\Omega_T)$ be a smooth cut-off function, such that $0\le \zeta \le 1$. Then for any $k\in \re$ the following holds.
  \begin{align*}
    \int_0^T \int_\Omega u^{m-1} |\nabla (u-k)_-|^2 \zeta^2 \dx x \dx t + \esssup_{t\in(0,T)} \int_\Omega \zeta^2 (u-k)_-^2 \dx x\\
\le C \left ( \int_0^T \int_\Omega (u-k)_-^2 \zeta |\zeta_t| \dx x \dx t + \int_0^T \int_\Omega u^{m-1} (u-k)_-^2 |\nabla \zeta |^2 \dx x \dx t\right ).
  \end{align*}
\end{lemma}

\begin{proof}
Take $\tau_1,\tau_2\in (0,T)$ such that $\tau_1<\tau_2$ and let $u^*$ denote the averaged function, defined in \eqref{averaged function}.  Take a test function $\ph=(u^*-k)_- \zeta^2$ in \eqref{averaged eq}, where $\zeta$ is a smooth cut-off function, such that $0\le \zeta \le 1$. Thus we have
\begin{align*}
0\le  &\int_{\tau_1}^{\tau_2} \int_\Omega   (\nabla u^m)^* \cdot \nabla(u^*-k)_-\zeta^2 \dx x \dx t  \\ +& \int_{\tau_1}^{\tau_2} \int_\Omega 2 (u^*-k)_- \zeta (\nabla u^m)^* \cdot \nabla \zeta \dx x \dx t
+ \int_{\tau_1}^{\tau_2} \int_\Omega(u^*-k)_-\zeta^2 \frac{\partial u^*}{\partial t}  \dx x \dx t\\
=& I_1 + I_2 + I_3.
\end{align*}
We observe, that
\begin{align*}
I_3&=\int_{\tau_1}^{\tau_2} \int_\Omega (u^*-k)_- \zeta^2 \frac{\partial u^*}{\partial t}  \dx x \dx t = -\frac 1 2 \int_{\tau_1}^{\tau_2} \int_\Omega  \zeta^2 \frac{\partial}{\partial t}(u^*-k)_-^2  \dx x \dx t\\
&=\frac 1 2  \int_\Omega  \zeta(x,\tau_1)^2 (u(x,\tau_1)^*-k)_-^2  \dx x-\frac 1 2  \int_\Omega  \zeta(x,\tau_2)^2 (u(x,\tau_2)^*-k)_-^2  \dx x \\
&+ \int_{\tau_1}^{\tau_2} \int_\Omega (u^*-k)_-^2 \zeta \zeta_t \dx x \dx t. 
\end{align*}
Thus, we may let $\sigma\rightarrow 0$ to get
\begin{align*}
  I_3\rightarrow &\frac 1 2  \int_\Omega  \zeta(x,\tau_1)^2 (u(x,\tau_1)-k)_-^2  \dx x-\frac 1 2  \int_\Omega  \zeta(x,\tau_2)^2 (u(x,\tau_2)-k)_-^2  \dx x \\
&+ \int_{\tau_1}^{\tau_2} \int_\Omega (u-k)_-^2 \zeta \zeta_t \dx x \dx t. 
\end{align*}
Now, we may write $I_1$ as 
\begin{align*}
I_1&\rightarrow \int_{\tau_1}^{\tau_2} \int_\Omega  m u^{m-1} \nabla u \cdot \nabla(u-k)_-\zeta^2 \dx x \dx t  \\
&= - \int_{\tau_1}^{\tau_2} \int_\Omega  m u^{m-1} | \nabla(u-k)_-|^2\zeta^2 \dx x \dx t.  
\end{align*}
Finally, we will use Young's inequality to control $I_2$. We have
\begin{align*}
\lim_{\sigma\rightarrow 0}  |I_2|&\le 2m\int_{\tau_1}^{\tau_2} \int_\Omega u^{m-1}(u-k)_- |\nabla (u-k)_-| \zeta|\nabla \zeta| \dx x \dx t\\
&\le \frac m 2 \int_{\tau_1}^{\tau_2} \int_\Omega u^{m-1} |\nabla (u-k)_-|^2 \zeta^2 \dx x \dx t \\
&+ 2m \int_{\tau_1}^{\tau_2} \int_\Omega u^{m-1} (u-k)_-^2 |\nabla \zeta |^2 \dx x \dx t.
\end{align*}
Collecting the facts, we get 
\begin{align}\label{almost caccioppoli}
  &\int_{\tau_1}^{\tau_2} \int_\Omega   u^{m-1} | \nabla(u-k)_-|^2\zeta^2 \dx x \dx t + \int_\Omega  \zeta(x,\tau_2)^2 (u(x,\tau_2)-k)_-^2  \dx x\nonumber\\
&-\int_\Omega  \zeta(x,\tau_1)^2 (u(x,\tau_1)-k)_-^2  \dx x \nonumber\\
&\le C\left ( \int_{\tau_1}^{\tau_2} \int_\Omega (u-k)_-^2 \zeta |\zeta_t| \dx x \dx t+ \int_{\tau_1}^{\tau_2} \int_\Omega u^{m-1} (u-k)_-^2 |\nabla \zeta |^2 \dx x \dx t\right).\nonumber\\
\end{align}
Taking the supremum over $\tau_2$ and letting $\tau_1\rightarrow 0$ concludes the proof.
\end{proof}

\begin{remark}
By choosing a test function $\ph_j=\zeta \eta_j$ in \eqref{averaged eq}, where $\zeta\in C_0^\infty(\Omega)$ and $\eta_j\in C_0^\infty ( -\eps, \tau+\eps)$, such that $\eta_j\rightarrow \chi_{[0,\tau]}$, we may integrate by parts in the time variable and let $j\rightarrow 0$ to get the inequality
\begin{equation}\label{integration by parts in time}
  \int_\Omega u(x,\tau) \zeta(x)\dx x \ge \int_\Omega u(x,0) \zeta (x)\dx x - \int_0^\tau \int_\Omega |\nabla u^m \cdot \nabla \zeta| \dx x \dx t.
\end{equation}
\end{remark}

\begin{lemma}\label{caccioppoli2}
  Let $u$ be a weak supersolution in $\Omega_T$, such that $u > 0$ and let $\zeta \in C_0^\infty (\Omega_T)$ be a smooth cut-off function, such that $0\le \zeta \le 1$. Let $\eps>0, \eps\ne 1$. Then $u$ satisfies
\begin{align*}
&\int_{0}^{T} \int_\Omega m u^{m-\eps-2} \zeta^2 |\nabla u|^2 \dx x \dx t + 
\frac 1 {\eps |1-\eps| } \esssup_{t\in (0,T)} \int_\Omega u^{1-\eps} \zeta^2 \dx x \\
&\le \frac{C_1m}{\eps^2} \int_{0}^{T} \int_\Omega u^{m-\eps} |\nabla \zeta|^2 \dx x \dx t + \frac{C_2}{\eps|1-\eps|} \int_{0}^{T} \int_\Omega u^{1-\eps} \zeta |\zeta_t| \dx x \dx t.
\end{align*}
\end{lemma}

\begin{proof}
 As in the proof of Lemma \ref{caccioppoli}, take $\tau_1,\tau_2\in (0,T)$ such that $\tau_1<\tau_2$ and let $u^*$ denote the averaged function, defined in \eqref{averaged function}.

 For $\lambda>0$, we define the dampening function
\[
H_\lambda(s)=
\begin{cases}
  \lambda^{-\eps}+\eps \lambda^{-1-\eps}(\lambda-s), \quad \text{if } 0 \le s \le \lambda,\\
s^{-\eps}, \quad \text{if } s>\lambda.
\end{cases}
\]
Now
\[
H'_\lambda(s)=
\begin{cases}
  -\eps\lambda^{-1-\eps}, \quad \text{if } 0 \le s \le \lambda,\\
-\eps s^{-1-\eps}, \quad \text{if } s>\lambda.
\end{cases}
\]
Thus $H_\lambda(s)$ is continuously differentiable. We denote
\[
h_\lambda(s)=
\begin{cases}
  \int_0^s H_\lambda(r) \dx r, \quad \text{if } \eps <1,\\
\int_s^\infty H_\lambda(r) \dx r, \quad \text{if } \eps >1.
\end{cases}
\] 
We choose a test function $\ph =  H_\lambda(u^*) \zeta^2$ in \eqref{averaged eq}, where $\zeta\in C_0^\infty(\Omega_T)$ is a smooth cut-off function, such that $0\le \zeta \le 1$. In order to let $\sigma\rightarrow 0$, we have to deal with the time derivative of $u^*$ appearing in \eqref{averaged eq}. We integrate by parts to get 
\begin{align*}
 & \int_{\tau_1}^{\tau_2} \int_\Omega \ph \frac {\partial u^*}{\partial t} \dx x \dx t =   \int_{\tau_1}^{\tau_2} \int_\Omega  H_\lambda(u^*) \frac {\partial u^*}{\partial t} \zeta^2\dx x \dx t = \iota \int_{\tau_1}^{\tau_2} \int_\Omega  \frac {\partial h_\lambda(u^*)}{\partial t} \zeta^2\dx x \dx t  \\
&= - \iota \int_{\tau_1}^{\tau_2} \int_\Omega h_\lambda (u^*) (\zeta^2)_t \dx x \dx t\\
& +  \iota \int_\Omega h_\lambda(u^*(x,{\tau_2})) \zeta(x,{\tau_2})^2 \dx x - \iota\int_\Omega h_\lambda(u^*(x,\tau_1)) \zeta(x,\tau_1)^2 \dx x.
\end{align*}
Here $\iota=\sgn(1-\eps)$. Now we may let $\sigma\rightarrow 0$ to obtain 
\begin{align}\label{sigma goes to zero}
 & \int_{\tau_1}^{\tau_2} \int_\Omega m u^{m-1} \nabla u \cdot \nabla \ph - \iota \int_{\tau_1}^{\tau_2} \int_\Omega  h_\lambda(u) (\zeta^2)_t \dx x \dx t \nonumber\\
&+ \iota \int_\Omega h_\lambda(u(x,{\tau_2})) \zeta(x,{\tau_2})^2 \dx x - \iota\int_\Omega h_\lambda(u(x,\tau_1)) \zeta(x,\tau_1)^2 \dx x \ge 0.\nonumber\\
\end{align}

We observe
\[
 m u^{m-1} \nabla u \cdot \nabla \ph = m u^{m-1} H'_\lambda(u)|\nabla u|^2 \zeta^2 +  2mu^{m-1}H_\lambda(u)\zeta \nabla \zeta\cdot \nabla u.
\]
We may estimate the second term on the right hand side by 
\begin{align*}
 & |2mu^{m-1}H_\lambda(u)\zeta \nabla \zeta\cdot \nabla u| \le 2mu^{m-1}H_\lambda(u)\zeta |\nabla \zeta ||\nabla u|\\
&=  m \Big(|\nabla u| \zeta u^{\frac{m-1}{2}} |H'_\lambda(u)|^{\frac 1 2 }\Big)\left(2 \frac {H_\lambda(u)}{|H'_\lambda(u)|^{\frac 12}} u^{\frac{m-1}2} |\nabla \zeta|\right)\\
&\le \frac{m}2 \left ( u^{m-1}|H'_\lambda(u)| \zeta^2 |\nabla u|^2 + 4 \frac {H_\lambda(u)^2}{|H'_\lambda(u)|} u^{m-1} |\nabla \zeta|^2\right). 
\end{align*}
Thus, we have the estimate
\begin{align*}
& -m u^{m-1}\nabla u \cdot \nabla \ph \\
&\ge -m  u^{m-1} H'_\lambda(u)|\nabla u|^2 \zeta^2 -  2m u^{m-1}H_\lambda(u)\zeta \nabla \zeta\cdot \nabla u \\
& \ge m u^{m-1} |H'_\lambda(u)||\nabla u|^2 \zeta^2 - | 2m u^{m-1}H_\lambda(u)\zeta \nabla \zeta\cdot \nabla u |\\
&\ge \frac{m}{2} u^{m-1} |H'_\lambda(u)||\nabla u|^2 \zeta^2 - 2m\frac {H_\lambda(u)^2}{|H'_\lambda(u)|} u^{m-1} |\nabla \zeta|^2.
\end{align*} 
Using this estimate in \eqref{sigma goes to zero} gives
\begin{align*}
 & \int_{\tau_1}^{\tau_2} \int_\Omega mu^{m-1} |H'_\lambda(u)| |\nabla u|^2 \zeta^2 \dx x \dx t + \iota\int_\Omega h_\lambda(u(x,\tau_1)) \zeta(x,\tau_1)^2 \dx x \\
&- \iota \int_\Omega h_\lambda (u(x,{\tau_2})) \zeta(x,{\tau_2})^2 \dx x\\
&\le C \int_{\tau_1}^{\tau_2} \int_\Omega \left(m\frac {H_\lambda(u)^2}{|H'_\lambda(u)|} u^{m-1} |\nabla \zeta|^2 + |h_\lambda |\zeta|\zeta_t| \right) \dx x \dx t.  \\
\end{align*}
Finally, we will show, that we get the correct Caccioppoli estimate as $\lambda\rightarrow 0$. We note that $H_\lambda(s), H'_\lambda(s)$ and $h_\lambda(s)$ are decreasing with respect to $\lambda$. Moreover, if $\eps<1$ 
\begin{align*}
  \lim_{\lambda \rightarrow 0} h_\lambda(u) =   \lim_{\lambda \rightarrow 0} \int_0^u H_\lambda(s) \dx s = \int_0^u   \lim_{\lambda \rightarrow 0} H_\lambda(s) \dx s = \frac{1}{1-\eps} u^{1-\eps}
\end{align*}
and if $\eps>1$
\begin{align*}
\lim_{\lambda \rightarrow 0} h_\lambda(u) =   \lim_{\lambda \rightarrow 0} \int_u^\infty H_\lambda(s) \dx s = \int_u^\infty   \lim_{\lambda \rightarrow 0} H_\lambda(s) \dx s = \frac{-1}{1-\eps} u^{1-\eps}.
\end{align*}
Here we used the monotone convergence theorem. We conclude $h_\lambda(u)\rightarrow \frac 1 {|1-\eps|}u^{1-\eps}$ as $\lambda \rightarrow 0$. Again, by the monotone convergence theorem, we conclude
\begin{align*}
&\int_{\tau_1}^{\tau_2} \int_\Omega mu^{m-1} |H'_\lambda(u)| |\nabla u|^2 \zeta^2 \dx x \dx t \rightarrow \eps \int_{\tau_1}^{\tau_2} \int_\Omega mu^{m-2-\eps} |\nabla u|^2 \zeta^2 \dx x \dx t, \\
&\int_{\tau_1}^{\tau_2} \int_\Omega |h_\lambda(u)| \zeta|\zeta_t| \dx x \dx t \rightarrow \frac{1}{|1-\eps|}\int_{\tau_1}^{\tau_2} \int_\Omega  u^{1-\eps}\zeta|\zeta_t| \dx x \dx t \quad \text{and}\\
&\iota \int_\Omega h_\lambda(u(x,{\tau_i}))\zeta(x,{\tau_i})^2 \dx x \rightarrow \frac \iota {|1-\eps|} \int_\Omega u(x,{\tau_i})^{1-\eps}\zeta(x,{\tau_i})^2 \dx x, \quad i\in \{1,2\}.\\
\end{align*}
In order to control the term involving $\frac {H_\lambda(u)^2}{|H'_\lambda(u)|}$, we observe
\begin{align*}
  \frac {H_\lambda(u)^2}{|H'_\lambda(u)|}=
  \begin{cases}
    \left( \frac {\lambda^{\frac{1-\eps}{2}}}{\sqrt{\eps}} + \sqrt{\eps}\lambda^{\frac{-1-\eps}{2}}(\lambda-u)\right)^2, \quad \text{if } 0\le u \le\lambda\\
\frac{u^{1-\eps}}{\eps}, \quad \text{if } u > \lambda.
  \end{cases}
\end{align*}
Therefore, if $\eps<1$, 
\[
\frac {H_\lambda(u)^2}{|H'_\lambda(u)|} \le \frac {H_1(u)^2}{|H'_1(u)|}
\]
and if $\eps>1$, 
\[
\frac {H_\lambda(u)^2}{|H'_\lambda(u)|} \le \frac{u^{1-\eps}}{\eps}.
\]
We may assume
\begin{equation} \label{finiteness assumption}
\int_{\tau_1}^{\tau_2} \int_\Omega u^{m-\eps} |\nabla \zeta|^2 \dx x \dx t < \infty,
\end{equation}
and therefore use the dominated convergence theorem to conclude
\[
\int_{\tau_1}^{\tau_2} \int_\Omega m\frac {H_\lambda(u)^2}{|H'_\lambda(u)|} u^{m-1} |\nabla \zeta|^2 \dx x \dx t \rightarrow \frac 1 \eps\int_{\tau_1}^{\tau_2} \int_\Omega m u^{m-\eps} |\nabla \zeta|^2 \dx x \dx t.
\]
We note, that if the integral in \eqref{finiteness assumption} is infinite, the estimate holds. Collecting the facts gives
\begin{align}\label{almost caccioppoli2}
   & \eps \int_{\tau_1}^{\tau_2} \int_\Omega mu^{m-2-\eps} |\nabla u|^2 \zeta^2 \dx x \dx t + \frac \iota {|1-\eps|}\int_\Omega u(x,\tau_1)^{1-\eps} \zeta(x,\tau_1)^2 \dx x \nonumber\\
&-\frac \iota {|1-\eps|}\int_\Omega u(x,{\tau_2})^{1-\eps} \zeta(x,{\tau_2})^2 \dx x\nonumber\\
&\le \frac {C_1} \eps\int_{\tau_1}^{\tau_2}\int_\Omega  m u^{m-\eps} |\nabla \zeta|^2 \dx x \dx t + \frac {C_2} {|1-\eps|} \int_{\tau_1}^{\tau_2}\int_\Omega  u^{1-\eps}\zeta|\zeta_t| \dx x \dx t. \nonumber \\
\end{align}
Take $\delta>0$. There exists $\tilde\tau \in (0,T)$, such that 
\[
\int_\Omega u(x,\tilde\tau)^{1-\eps} \zeta(x,\tilde\tau)^2 \dx x \ge \esssup_{t\in(0,T)}\int_\Omega u(x,\tilde\tau)^{1-\eps} \zeta(x,\tilde\tau)^2 \dx x - \delta.
\]
If $\eps<1$, we choose $\tau_1=\tilde\tau$ and let $\tau_2\rightarrow T$. On the other hand, if $\eps>1$, we choose $\tau_2=\tilde\tau$ and let $\tau_1\rightarrow 0$. Thus we get the estimate
\begin{align*}
   &\int_{0}^{T} \int_\Omega mu^{m-2-\eps} |\nabla u|^2 \zeta^2 \dx x \dx t +\frac 1 {\eps|1-\eps|} \esssup_{t\in(0,T)} \int_\Omega u^{1-\eps} \zeta^2 \dx x \\
&\le \frac {C_1} {\eps^2}\int_{0}^{T}\int_\Omega  m u^{m-\eps} |\nabla \zeta|^2 \dx x \dx t + \frac{C_2}{\eps|1-\eps|} \int_{0}^{T}\int_\Omega  u^{1-\eps}\zeta|\zeta_t| \dx x \dx t.  \\
\end{align*}
\end{proof}

\begin{remark}
Inequalities \eqref{almost caccioppoli} and \eqref{almost caccioppoli2} hold even if the cut-off function at hand is not compactly supported in the time variable. This will be useful later on. For instance, it allows us to use cut-off functions, which vanish only on the parabolic boundary of the cylinder.  
\end{remark}
We now prove a measure theoretical lemma. The following lemma is a modification of the one in \cite{local clustering}, as we work with balls instead of cubes. A similar statement has been presented in the metric space setting in \cite{KMMP}. However, for our purposes it is essential to keep track of the dependence of $\eps$ on the various constants. Hence we give a detailed proof for this known result.  

\begin{lemma}\label{measure theoretical lemma}
  Let $u\in W^{1,1}(B(x_0,\rho))$. Suppose that
\begin{align*}
&||u||_{W^{1,1}(B(x_0,\rho))}\le h \rho^{n-1} \quad \text{and}\\
&|\{x\in B(x_0,\rho): u(x)>1\}| \ge \tilde \delta |B(x_0,\rho)|
\end{align*}
for some $h>0$ and $\tilde \delta\in(0,1)$. Then for every $\delta, \lambda \in (0,1)$ there exists $\tilde x\in B(x_0,\rho)$ and $\eps$, such that 
\[
|\{x\in B(\tilde x,\eps \rho) : u(x)>\lambda\}| > (1-\delta) |B(\tilde x,\eps\rho)|.
\]
Here $\eps =\frac {C(\lambda, \delta, n)}h \tilde \delta^2$.  

\end{lemma}

\begin{proof}
Take $\eps>0$. Define 
\[
\mathcal{F}=\{B(x,\eps \rho) : x\in B(x_0,(1-\eps)\rho)\}.
\]
By Besicovitch's covering lemma, there is a constant $c_n$ depending on the dimension $n$, and collections $\mathcal G_i \subset \mathcal F$, $i=1,\ldots c_n$, such that each $\mathcal G_i$ consists of disjoint balls and 
\[
B(x_0,(1-\eps)\rho) \subset \bigcup_{i=1}^{c_n} \bigcup_{B\in \mathcal G_i} B.
\] 
Moreover, the number of balls in the collection $\mathcal G = \bigcup_{i=1}^{c_n} \mathcal G_i$, denoted by $|\mathcal G|$, is at most $\frac {c_n}{\eps^n}$. We observe, that 
\begin{align}\label{(1-eps)rho}
  |\{x\in B(x_0,(1-\eps)\rho) : u(x)>1 \} &\ge   |\{x\in B(x_0,(1-\eps)\rho) : u(x)>1 \} \nonumber\\
&- |B(x_0,\rho) \setminus B(x_0,(1-\eps)\rho)|\nonumber\\
&\ge (\tilde \delta - (1-(1-\eps)^n))|B(x_0,(1-\eps)\rho)|\nonumber\\
&\ge \frac {\tilde \delta} 2 |B(x_0,(1-\eps)\rho)|,\nonumber \\
\end{align}
whenever $\eps \le \frac {\tilde \delta}{2n}$. Define the subcollections 
\begin{align*}
&  \mathcal G^+ = \left\{ B\in \mathcal G : |\{ x\in B : u(x)>1\} | > \frac {\tilde \delta}{8c_n} |B| \right\} \quad \text{and}\\
 & \mathcal G^- = \left\{ B\in \mathcal G : |\{ x\in B : u(x)>1\} | \le \frac {\tilde \delta}{8c_n} |B| \right\}.
\end{align*}
Since $\mathcal G$ is a covering of $B(x_0,(1-\eps)\rho)$, \eqref{(1-eps)rho} implies
\begin{align*}
&  \sum_{B_i\in \mathcal G^+} |\{x\in B_i : u(x)>1\}| +   \sum_{B_j\in \mathcal G^-} |\{x\in B_j : u(x)>1\}|\\
&\ge |\{x\in B(x_0,(1-\eps)\rho) : u(x)>1 \} \ge \frac{\tilde \delta }{2} |B(x_0,(1-\eps)\rho)| \ge \frac {\tilde \delta }{4 \eps^n} |B|
\end{align*}
for every $B\in \mathcal G$. Here we assumed $(1-\eps)^n\ge \frac 12$. Thus,
\begin{align*}
  \frac {\tilde \delta}{4\eps^n} &\le \sum_{B_i\in \mathcal G^+} \frac{|\{x\in B_i : u(x)>1\}|}{|B_i|} +   \sum_{B_j\in \mathcal G^-} \frac{|\{x\in B_j : u(x)>1\}|}{|B_j|}\\
&\le |\mathcal G^+|+  \frac{\tilde \delta}{8 c_n} (|\mathcal G| - |\mathcal G^+|)\le \left( 1- \frac{\tilde \delta}{8 c_n}\right) |\mathcal G^+| + \frac{\tilde \delta}{8 \eps^n}.
\end{align*}
 We get a lower bound for the number of cubes in $\mathcal G^+$,
 \begin{equation}
   \label{G+ lower bound}
|\mathcal G^+| > \frac{\tilde \delta c_n}{\eps^n(8c_n-\tilde \delta)}.
 \end{equation}
We fix the numbers $\delta, \lambda \in (0,1)$. Suppose, that for every $B\in \mathcal G^+$, we have
\begin{equation}
  \label{counter assumption}
|\{x\in B : u(x)>\lambda\}| \le (1-\delta)|B|.
\end{equation}
Then 
\begin{equation}\label{lambdabound}
 |\{x\in B: u(x)\le \lambda \}| > \delta |B|.
\end{equation}
The De Giorgi type lemma \cite[Lemma II]{degiorgi} implies 
\begin{align*}
 &(1-\lambda) |\{x\in B: u(x)\le \lambda \}||\{x\in B: u(x)> 1 \}| \le C (\eps \rho)^{n+1} \int_{B} |\nabla u| \dx x \\
\end{align*}
Thus, using \eqref{lambdabound} and the fact, that $B\in \mathcal G^+$, gives the estimate
\begin{align*}
  \frac{(1-\lambda)\delta \tilde \delta}{8c_n} |B|^2 \le C (\eps \rho)^{n+1} \int_{B} |\nabla u| \dx x. 
\end{align*}
We sum over $\mathcal G^+$ and use \eqref{G+ lower bound} to get
\begin{align*}
    \frac{c(\lambda, \delta, n) \tilde \delta^2}{8c_n-\tilde \delta} \frac {\rho^{n-1}} \eps \le C c_n \int_{B(x_0,\rho)} |\nabla u| \dx x \le C c_n h \rho^{n-1}. 
\end{align*}
We may choose 
\[
\eps = \frac {C(\lambda, \delta, n)}h \tilde \delta^2,
\]
thus leading to contradiction. Therefore \eqref{counter assumption} does not hold. That is, there exists $B\in \mathcal G^+$, such that 
\[
|\{x\in B: u(x)>\lambda\}| > (1-\delta) |B|.
\]
\end{proof}

\section{Expansion of positivity}
In this chapter, we will show that the supersolutions satisfy the ``expansion of positivity''. That is, if the supersolution is large in a large portion of a ball at some time, then the supersolution will be bounded away from zero in a larger space-time cylinder after some waiting time. This is the statement of Lemma \ref{expansion of positivity}. 

First, we will show, that if the supersolution $u$ is large in a large portion of a ball at some time level, then $u$ remains bounded away from zero in a large portion of the ball up to some time level.
\begin{lemma}\label{time propagation}
  Let $k>0$ and $\gamma\in (0,1)$. Suppose that $u$ is a weak supersolution in $\Omega_{T_0}$, such that
\[
|\{x\in B(x_0,\rho) : u(x,s)>k\}| \ge \gamma |B(x_0,\rho)| 
\]
for some $s\in(0,\rho^2)$. There exists a constant $C=C(m,n)$, such that 
\[
\left| \left\{x\in B(x_0,\rho) : u(x,t)>\frac{\gamma}{8}k\right \}\right| \ge \frac \gamma 8 |B(x_0,\rho)|
\]
for almost every $t\in \left(s, s+\frac{\gamma^2 \rho^2}{C k^{m-1}}\right]$.
\end{lemma}

\begin{proof}
  Denote $T=\frac{\gamma^2 \rho^2}{C k^{m-1}}$ and $\eps = \frac \gamma {5n}$. Let $\zeta\in C_0^\infty (B(x_0,(1+\eps)\rho))$ be a cuf-off function, such that 
\[
\begin{cases}
  \zeta=1 \quad \text{in } B(x_0, \rho),\\
0\le \zeta \le 1 \quad \text{and} \\
|\nabla \zeta| \le \frac{C_1}{\eps  \rho}.
\end{cases}
\]
Take $\tau_0,\tau\in (s,s+T)$, such that $\tau_0<\tau$. Since $u$ is a non-negative weak supersolution, \eqref{almost caccioppoli} gives
\begin{align*}
    &\int_{\tau_0}^{\tau} \int_{B(x_0,(1+\eps)\rho)}  u^{m-1} | \nabla(u-k)_-|^2\zeta^2 \dx x \dx t \\
&+ \int_{B(x_0,(1+\eps)\rho)}  \zeta(x)^2 (u(x,\tau)-k)_-^2  \dx x\\
&\le C\int_{\tau_0}^\tau \int_{B(x_0,(1+\eps)\rho)} u^{m-1}(u-k)_-^2 |\nabla \zeta |^2 \dx x \dx t \\
&+ \int_{B(x_0,(1+\eps)\rho)}  \zeta(x)^2 (u(x,\tau_0)-k)_-^2  \dx x. 
\end{align*}
This implies 
\begin{align*}
   &\int_{B(x_0,(1+\eps)\rho)} \zeta(x)^2 (u(x,\tau)-k)_-^2  \dx x\\
&\le C\int_{\tau_0}^\tau \int_{B(x_0,(1+\eps)\rho)} u^{m-1}(u-k)_-^2 |\nabla \zeta |^2 \dx x \dx t \\
&+ \int_{B(x_0,(1+\eps)\rho)}  \zeta(x)^2 (u(x,\tau_0)-k)_-^2  \dx x. 
\end{align*}
Thus letting $\tau_0\rightarrow s$ gives
\begin{align}\label{level set estimate}
   &\int_{B(x_0,\rho)}  \zeta(x)^2 (u(x,\tau)-k)_-^2  \dx x \nonumber\\
&\le C\int_{s}^{s+T} \int_{B(x_0,(1+\eps)\rho)} u^{m-1}(u-k)_-^2 |\nabla \zeta |^2 \dx x \dx t \nonumber\\
&+ \int_{B(x_0,(1+\eps)\rho)}  \zeta(x)^2 (u(x,s)-k)_-^2  \dx x. \nonumber\\
\end{align}
This holds for almost every $\tau\in(s,s+T)$. We write
\begin{align}\label{1eps split}
&\int_{B(x_0,(1+\eps)\rho)}  \zeta(x)^2 (u(x,s)-k)_-^2  \dx x \nonumber\\
&= \int_{B(x_0,(1+\eps)\rho)\setminus B(x_0,\rho)}  \zeta(x)^2 (u(x,s)-k)_-^2  \dx x \nonumber\\
&+\int_{B(x_0,\rho)}  \zeta(x)^2 (u(x,s)-k)_-^2  \dx x.\nonumber\\
\end{align}
The first term on the right hand is bounded from above by
\begin{align*}
  &\int_{B(x_0,(1+\eps)\rho)\setminus B(x_0,\rho)}  \zeta(x)^2 (u(x,s)-k)_-^2  \dx x \\
&\le k^2 |B(x_0,(1+\eps)\rho)\setminus B(x_0,\rho)| =
k^2((1+\eps)^n-1)|B(x_0,\rho)|.  
\end{align*}
By the assumption on $u$, the second term on the right hand side of \eqref{1eps split} is bounded from above by 
\begin{align*}
  \int_{B(x_0,\rho)}  \zeta(x)^2 (u(x,s)-k)_-^2  \dx x&\le k^2|\{x\in B(x_0,\rho) : u(x,s)\le k\}| \\
&\le (1-\gamma)k^2 |B(x_0,\rho)|.
\end{align*}
Thus,
\begin{align*}
  &\int_{B(x_0,(1+\eps)\rho)}  \zeta(x)^2 (u(x,s)-k)_-^2  \dx x \\
&\le (((1+\eps)^n-1)+(1-\gamma))k^2 |B(x_0,\rho)| \le \left(1-\frac {3\gamma} 4\right )k^2|B(x_0,\rho)|.
\end{align*}
Here we used the fact 
\[
(1+\eps)^n-1 \le \frac {n\eps}{1-n\eps} \le \frac \gamma 4.
\]
The first term on the right hand side of \eqref{level set estimate} can be estimated using the boundedness of $|\nabla \zeta|$ and the fact that $\supp \nabla \zeta \subset B(x_0,(1+\eps)\rho)\setminus B(x_0,\rho)$. We have
\begin{align*}
 & \int_{s}^{s+T} \int_{B(x_0,(1+\eps)\rho)} u^{m-1}(u-k)_-^2 |\nabla \zeta |^2 \dx x \dx t \\
&\le C_2 k^{m+1} \left (\frac {C_1}{\eps \rho}\right)^2 T |B(x_0,(1+\eps)\rho)\setminus B(x_0,\rho)|\\
&\le C_2 k^{m+1} \left (\frac {C_1}{\eps \rho}\right)^2 T \frac \gamma 4 |B(x_0,\rho)| =  \frac{C_3 n^2 k^{m+1}}{\rho^2\gamma}T |B(x_0,\rho)|\\
&=\frac{C_3 n^2 k^2 \gamma} C |B(x_0,\rho)| \le \frac {\gamma k^2} 4 |B(x_0, \rho)|. 
\end{align*}
Here $C$ is chosen in such a way, that $C\ge 4C_3n^2$. To get a lower bound for the left hand side of \eqref{level set estimate}, we observe that 
\[
\left\{x\in B(x_0,\rho): u(x,t)\le \frac \gamma 8 k \right\} = \left\{x\in B(x_0,\rho): (u(x,t)-k)_-\ge \left(1-\frac \gamma 8\right) k \right\}
\]
and use Chebyshev's inequality to obtain
\begin{align*}
&\left(1-\frac \gamma 8\right)^2 k^2\left |\left\{x\in B(x_0,\rho): u(x,t)\le \frac \gamma 8 k \right\}\right| \\
&\le \int_{B(x_0,\rho)}  \zeta(x)^2 (u(x,\tau)-k)_-^2  \dx x.
\end{align*}
Thus, collecting the facts gives
\[
  \left(1-\frac \gamma 8\right)^2 \left|\left\{x\in B(x_0,\rho): u(x,t)\le \frac \gamma 8 k \right\}\right|\le \left( 1- \frac \gamma 2 \right) |B(x_0,\rho)|.
\]
Approximating
\[
\frac{ \left( 1-\frac \gamma 2 \right ) }{\left( 1 - \frac \gamma 8 \right )^2} \le \frac {1- \frac \gamma 2}{1-\frac \gamma 4} \le 1- \frac \gamma 8
\]
then gives
\[
  |\{x\in B(x_0,\rho): u(x,t)> \frac \gamma 8 k \}|\ge \frac \gamma 8 |B(x_0,\rho)|,
\]
thus concluding the proof.
\end{proof}

We will make a change of variables, in order to extend the positivity set even further in time. We consider the function $w(x,\tau)=g(\Lambda^{-1}(\tau))u(x,\Lambda^{-1}(\tau))$, which is a supersolution in a space-time cylinder $\Omega_{T_0'}$. The key idea here is to compensate the decay of $u$ by a factor $g(\Lambda^{-1}(\tau))$, thus allowing us to consider the times beyond the threshold given by Lemma \ref{time propagation}.

\begin{lemma}[Change of Variables]\label{change of variables}
Let $u$ be a non-negative weak supersolution in $\Omega_{T_0}$, such that
\[
|\{x\in B(x_0,\rho): u(x,s)> k\} | \ge \gamma |B(x_0,\rho)|. 
\]
Then
\[
w(x,\tau)=\frac{e^{\frac \tau {m-1}}} k (\delta \rho^2)^{\frac 1 {m-1}} u \left(x,s+\frac {e^\tau}{k^{m-1}} \delta \rho^2 \right)
\]  
is a non-negative supersolution in $\Omega_{T_0'}$, where $T_0'=\ln\left((T_0-s)\frac{k^{m-1}}{\delta \rho^2}\right)$. Moreover, $w$ satisfies
\[
|\{x\in B(x_0,\rho): w(x,\tau)> k_0 \} | \ge \frac \gamma 8 |B(x_0,\rho)|
\]
for almost every $\tau\ge 0$. Here $\delta=\frac{\gamma^2}C$ and $k_0=\frac \gamma 8 (\delta \rho^2)^{\frac 1 {m-1}}$. 
\end{lemma}
\begin{proof}
  First we will show that $w$ is indeed a weak supersolution. Define 
\[
g(t) = (t-s)^{\frac 1 {m-1}} \quad \text{and} \quad \Lambda(t)=\ln \left( (t-s) \frac {k^{m-1}}{\delta \rho^2}\right).
\]
Then $w$ can be written as $w(x,\tau)=g(\Lambda^{-1}(\tau))u(x,\Lambda^{-1}(\tau))$. Let $\ph\in C_0^\infty(\Omega_{T_0'})$ be a non-negative test function and define
\begin{align*}
 & \eta(x,t)=g(t)\ph(x,\Lambda(t))= g(t)\tilde{\ph}(x,t) \quad \text{and}\\
&\tilde{w}(x,t)=w(x,\Lambda(t))= g(t)u(x,t).
\end{align*}
Now $\nabla u^m\cdot \nabla \eta$ and $u\eta_t$ can be written in terms of $\tilde{w}, \tilde{\ph}$ and $\Lambda'(t)$ as 
\begin{align*}
&  \nabla u^m \cdot \nabla \eta= \frac{1}{g^m} \nabla \tilde{w}^m \cdot \nabla \eta = \frac 1 {t-s} \nabla \tilde{w}^m\cdot \nabla \tilde{\ph} = \Lambda'(t)  \nabla \tilde{w}^m\cdot \nabla \tilde{\ph} \quad \text{and}\\
&u \eta_t = g'(t)u \tilde{\ph} + \tilde{w}  \tilde {\ph}_t = \frac {g'}g \tilde{w} \tilde {\ph} + \tilde{w} \tilde {\ph}_t = \frac{\Lambda'(t)}{m-1}\tilde{w} \tilde {\ph} + \tilde{w}\tilde {\ph}_t.
\end{align*}
Denote $T_1=s+\frac{\delta \rho^2}{k^{m-1}}$. Since $u$ is a weak supersolution, it satisfies
\begin{align*}
&0\le \int_{T_1}^{T_0} \int_\Omega \Big(-u \eta_t + \nabla u^m \cdot \nabla \eta \Big)\dx x \dx t\\
&=-\int_{T_1}^{T_0} \int_\Omega \left(\frac{\Lambda'(t)}{m-1}\tilde{w} \tilde {\ph} + \tilde{w}\tilde {\ph}_t \right)\dx x \dx t + \int_{T_1}^{T_0} \int_\Omega \Lambda'(t)  \nabla \tilde{w}^m\cdot \nabla \tilde{\ph} \dx x \dx t.
\end{align*}
Recalling the definition of $\tilde w$ and $\tilde \ph$ gives
\begin{align*}
0&\le -\int_{T_1}^{T_0} \int_\Omega \left (\frac{\Lambda'(t)}{m-1}w(x,\Lambda(t)) \ph(x,\Lambda(t)) + w(x,\Lambda(t))\ph(x,\Lambda(t))_t \right) \dx x \dx t \\
&+ \int_{T_1}^{T_0} \int_\Omega \Lambda'(t)  \nabla w(x,\Lambda(t))^m\cdot \nabla \ph(x,\Lambda(t)) \dx x \dx t\\
&=-\int_{0}^{T_0'} \int_\Omega \left(\frac 1{m-1}w \ph + w\ph_\tau \right)\dx x \dx \tau + \int_{0}^{T_0'} \int_\Omega  \nabla w^m\cdot \nabla \ph \dx x \dx \tau\\
&\le \int_{0}^{T_0'} \int_\Omega \Big(- w\ph_\tau + \nabla w^m\cdot \nabla \ph \Big)\dx x \dx \tau,
\end{align*}
thus showing that $w$ is a weak supersolution. The next step is to show, that $w$ satisfies the inequality
\[
|\{x\in B(x_0,\rho): w(x,\tau)> k_0 \} | \ge \frac \gamma 8 |B(x_0,\rho)|.
\]
Take $\sigma\le 1$. By assumption
\[
|\{x\in B(x_0,\rho): u(x,s)>\sigma k\} | \ge|\{x\in B(x_0,\rho): u(x,s)> k\} | \ge \gamma |B(x_0,\rho)|. 
\]
By Lemma \ref{time propagation} we have
\[
\left|\left\{x\in B(x_0,\rho): u(x,t)> \frac \gamma 8 \sigma k\right\} \right| \ge \frac \gamma 8 |B(x_0,\rho)|
\]
for almost every $t\in \left (s, s+ \frac {\delta \rho^2}{(\sigma k)^{m-1}}  \right ]$, where $\delta = \frac {\gamma^2}{C}$. Thus, in particular we have 
\[
\left |\left \{x\in B(x_0,\rho): u\left (x, s+ \frac {\delta \rho^2}{(\sigma k)^{m-1}}\right)> \frac \gamma 8 \sigma k \right\} \right| \ge \frac \gamma 8 |B(x_0,\rho)|.
\]
Choosing $\sigma=\sigma(\tau)=e^{-\frac \tau {m-1}}$ gives
\begin{align*}
&\left |\left \{x\in B(x_0,\rho): \frac{e^{\frac {\tau}{m-1}}}{k} (\delta \rho^2)^{\frac 1 {m-1}} u\left (x, s+ \frac {\delta \rho^2}{k^{m-1}} e^\tau \right)> \frac \gamma 8 (\delta \rho^2)^{\frac 1 {m-1}} \right\} \right| \\
&\ge \frac \gamma 8 |B(x_0,\rho)|
\end{align*}
for almost every $\tau\ge 0$. Now, denoting $k_0=\frac \gamma 8 (\delta \rho^2)^{\frac 1 {m-1}}$ and recalling the definition of $w$, shows that 
\[
|\{x\in B(x_0,\rho): w(x,\tau)> k_0 \} | \ge \frac \gamma 8 |B(x_0,\rho)|,
\] 
concluding the proof.
\end{proof}

The next lemma shows, that $w$ is small in a small portion of a space-time cylinder $B(x_0,4\rho)\times (T,\theta T)$. This, however, realizes after a waiting time $T$, depending on $\gamma$ and the size of the portion, where $w$ is small. Then, the expansion of positivity for $w$ follows from a De Giorgi type lemma for the weak supersolutions.

\begin{lemma}\label{nu lemma}
  Let $u$ be a weak supersolution in $\Omega_{T_0}$ and let $w$ be defined as in Lemma \ref{change of variables}. Then for every $\nu >0$ there exists $\eps>0$ and a time level $T$, such that 
\[
|\{ (x,t)\in B(x_0,4\rho)\times (T,\theta T): w< \eps k_0 \} | \le \nu |B(x_0,4\rho)\times (T,\theta T)|. 
\]
The dependence of $\eps$ and $T$ on the parameters $\nu$ and $\gamma$ can be traced as
\[
\eps=2^{-N} \quad \text {and} \quad T=\frac{2}{(\eps k_0)^{m-1}} (4\rho)^2, 
\]
where $ N=\left(\frac C {\gamma\nu} \right)^2+1$.
 The parameter $\theta \ge 2$ can be chosen as we please. 
\end{lemma}

\begin{proof}
  Let $k_j=2^{-j}k_0$ for $j=0,1,\ldots, N$ and $\eps = 2^{-N}$, where $N\in \n$ will be determined in terms of $\gamma$ and $\nu$. By the De Giorgi type lemma \cite[Lemma II]{degiorgi} the following holds 
\begin{align*}
&(k_j-k_{j+1})|\{x\in B(x_0,4\rho): w(x,t)<k_{j+1}\}| \\
&\le \tilde C \frac {\rho^{n+1}}{|\{x\in B(x_0,4\rho) : w(x,t)>k_j\}|} \int_{A_j(t)} |\nabla w| \dx x
\end{align*}
at each time level $t$. Here $A_j(t)=\{x\in B(x_0, 4\rho): k_{j+1} < w(x,t) < k_j \}$. By Lemma \ref{change of variables}, we have
\begin{align*}
&|\{x\in B(x_0,4\rho) : w(x,t)>k_j\}| \ge |\{x\in B(x_0,\rho) : w(x,t)>k_0\}|\\
&\ge \frac \gamma 8 |B(x_0,\rho)| = \gamma \tilde C \rho^n.
\end{align*}
Now, since $k_j-k_{j+1}=k_{j+1}$, we get the estimate
\[
|\{x\in B(x_0,4\rho): w(x,t)<k_{j+1}\}| \le \frac {C \rho}{\gamma k_{j+1}} \int_{A_j(t)} |\nabla w| \dx x.
\]
Integrating over the time interval $(T,\theta T)$ gives 
\begin{align}\label{nu levelset}
&|\{(x,t)\in B(x_0,4\rho)\times (T,\theta T): w(x,t)<k_{j+1}\}| \nonumber\\
&\le \frac {C \rho}{\gamma k_{j+1}} \int_T^{ \theta T}\int_{A_j(t)} |\nabla w| \dx x \dx t.\nonumber\\
\end{align}
In order to control the right hand side, we denote $A_j=\{(x,t) \in B(x_0,4\rho)\times (T,\theta T) : k_{j+1}<w<k_j\}$ and use H\"older's inequality to get
\begin{equation}\label{nu holder}
\iint_{A_j} |\nabla w | \dx x \dx t \le \left ( \iint_{A_j} |\nabla w| ^2 \dx x \dx t \right)^{1/2} |A_j|^{1/2}.
\end{equation}
Since $k_{j+1}<w<k_j$ in $A_j$, we may approximate 
\begin{align}\label{nu Aj}
  \iint_{A_j} |\nabla w|^2 \dx x \dx t &=   \iint_{A_j} |\nabla (w-k_j)_-|^2 \dx x \dx t\nonumber\\
&\le \frac 1 {k_{j+1}^{m-1}} \iint_{A_j} w^{m-1}  |\nabla (w-k_j)_-|^2 \dx x \dx t\nonumber\\
&\le \frac 1 {k_{j+1}^{m-1}} \int_T^{\theta T} \int_{B(x_0,4\rho)} w^{m-1}  |\nabla (w-k_j)_-|^2 \dx x \dx t.\nonumber \\
\end{align}
Let $\zeta \in C_0^\infty(B(x_0,8\rho)\times (0,T_0'))$ be a smooth cut-off function, such that $0\le \zeta \le 1$ and 
\[
\begin{cases}
  \zeta = 1 \quad \text{in } B(x_0,4\rho)\times (T,\theta T),\\
|\nabla \zeta| \le \frac {\tilde{C}}{4\rho} \quad \text{and}\\
|\zeta_t| \le \frac{\tilde{C}}{ T}.
\end{cases}
\]
 Using Lemma \ref{caccioppoli}, we get the estimate
\begin{align}\label{nu caccioppoli}
 & \frac 1 {k_{j+1}^{m-1}} \int_T^{\theta T} \int_{B(x_0,4\rho)} w^{m-1}  |\nabla (w-k_j)_-|^2 \dx x \dx t \nonumber\\
&\le \frac C {k_{j+1}^{m-1}} \int_0^{\theta T} \int_{B(x_0,8\rho)} \Big((w-k_j)_-^2 \zeta |\zeta_t| + w^{m-1} (w-k_j)_-^2 |\nabla \zeta|^2 \Big)\dx x \dx t \nonumber\\
&\le \frac C {k_{j+1}^{m-1}}\left (\frac {k_j^2} { T} + \frac{k_j^{m+1}}{(4\rho)^2}\right )|B(x_0,8\rho) \times (0,\theta T)|\nonumber\\
&\le  \frac {C k_j^{m+1}} {k_{j+1}^{m-1}(4\rho)^2}\left (\frac {2} {k_j^{m-1} (\eps k_0)^{m-1}} +1\right )|B(x_0,4\rho) \times (T,\theta T)|\nonumber\\
&\le \frac {C k_j^2} {(4\rho)^2}|B(x_0,4\rho) \times (T,\theta T)|\nonumber\\
\end{align}
Combining the estimates from \eqref{nu levelset}, \eqref{nu holder}, \eqref{nu Aj} and \eqref{nu caccioppoli} gives
\begin{align*}
&|\{(x,t)\in B(x_0,4\rho)\times (T, \theta T): w(x,t)<k_{j+1}\}| \\
&\le \frac C \gamma |A_j|^{1/2} |B(x_0,4\rho)\times (T, \theta T)|^{1/2}.
\end{align*}
Since, $k_N<k_{j+1}$ for $j=0,\ldots, N-1$, we have 
\begin{align*}
&|\{(x,t)\in B(x_0,4\rho)\times (T,\theta T): w(x,t)<k_N\}|^2 \\
&\le \frac C {\gamma^2} |A_j| |B(x_0,4\rho)\times (T, \theta T)|.
\end{align*}
By definition, the sets $A_j\subset B(x_0,4\rho)\times (T,\theta T)$ are disjoint, and therefore summing over $j$ gives 
\begin{align*}
&(N-1)|\{(x,t)\in B(x_0,4\rho)\times (T,\theta T): w(x,t)<k_N\}|^2 \\
&\le \frac C {\gamma^2} |B(x_0,4\rho)\times (T, \theta T)|^2.
\end{align*}
Hence, the result holds for $\nu = \frac {C}{\gamma \sqrt{N-1}}$ and $\eps=k_N$, where $N\in \n$ can be chosen as we please.
\end{proof}

We will prove a De Giorgi type lemma for the weak supersolutions. 

\begin{lemma}\label{degiorgi}
  Let $u$ be a non-negative, locally bounded weak supersolution in a neighbourhood of $U_{2\rho}=B(x_0,2\rho)\times (t_0, t_0+\lambda (2\rho)^2)$. Let $\xi,a \in(0,1)$ and let $\mu \ge \esssup_{U_{2\rho}} u$. Then, there exists a constant $\nu=\nu(a,\xi,\mu,\lambda,m,n)$, such that if
\[
|\{(x,t)\in U_{2\rho}: u(x,t)\le \xi \mu\}| \le \nu |U_{2\rho}|,
\]
then 
\[
u\ge a \xi \mu \quad \text{almost everywhere in } B(x_0,\rho)\times (t_0+3\lambda \rho^2 , t_0+4\lambda \rho^2).
\]
\end{lemma}

\begin{proof}
 Denote $T=t_0+4\lambda \rho^2$. Let $\rho_j = (1+2^{-j})\rho$, $T_j=T-\lambda \rho_j^2$, $B^j=B(x_0,\rho_j)$ and $U^j=B^j\times (T_j,T)$. Moreover, let $k_j=(2^{-j} + (1-2^{-j})a)\xi \mu$. Define a function $v=\max\{u, \frac 1 2 a \xi \mu\}$. We observe, that $k_j>\frac 1 2 a \xi \mu$, which implies 
\[
A_j=\{(x,t)\in U^j : v(x,t)<k_j \} = \{(x,t)\in U^j : u(x,t)<k_j \}.
\]
Thus, it suffices to show, that $|A_j|\rightarrow 0$ as $j\rightarrow \infty$. We will show this by using the fast geometric convergence lemma \cite[Lemma 7.1, p.220]{giusti}. 

On the set $A_{j+1}$ we have $v<k_{j+1}$ and therefore 
\[
(v-k_j)_-> k_j-k_{j+1} = \frac {1-a}{2^{j+1}}\xi \mu.
\]
Let $\zeta$ be a smooth cut-off function, such that $0\le\zeta\le 1$ and
\[
\begin{cases}
  \zeta=1 \quad \text{in } U^{j+1},\\
\zeta= 0 \quad \text{on } \bd_pU^j,\\
|\nabla \zeta| \le \frac 1 {\rho_j-\rho_{j+1}}= \frac {2^{k+1}}{\rho} \quad \text{and}\\
|\zeta_t|\le \frac 1 {T_j-T_{j+1}} \le \frac {2^{2(j+1)}} {\lambda \rho^2}.
\end{cases}
\]
Now
\begin{equation}\label{degiorgi Aj1}
  \left(\frac {1-a}{2^{j+1}} \right )^2 (\xi \mu)^2 |A_{j+1}| \le \iint_{U^{j+1}} (v-k_j)_-^2 \dx x \dx t \le  \iint_{U^{j}} (v-k_j)_-^2\zeta^2 \dx x \dx t.
\end{equation}
Using H\"older's inequality and the parabolic Sobolev's inequality \cite[Proposition 3.1]{dibenedetto} with $q=2 \frac{n+2}n$, $p=2$ and $m=2$, we get 
\begin{align}\label{embedding}
  & \iint_{U^{j}} (v-k_j)_-^2\zeta^2 \dx x \dx t\le \left (  \iint_{U^{j}} ((v-k_j)_-^2\zeta^2)^{(n+2)/n} \dx x \dx t\right)^{n/(n+2)} |A_j|^{2/(n+2)}\nonumber\\
&\le C\left (  \iint_{U^{j}} |\nabla((v-k_j)_-\zeta)|^2 \dx x \dx t\right)^{n/(n+2)}\nonumber\\
&\times \left ( \esssup_{t\in (T_j,T)} \int_{B^j}(v-k_j)_-^2 \zeta^2 \dx x \right )^{2/(n+2)}|A_j|^{2/(n+2)}.\nonumber\\
\end{align}
To find an upper bound for the right hand side, we observe
\begin{align*}
  &\left ( \frac {a\xi \mu}{2} \right)^{m-1} \iint_{U^j} |\nabla ((v-k_j)_-\zeta ) |^2 \dx x \dx t \le \iint_{U^j} v^{m-1}|\nabla ((v-k_j)_-\zeta ) |^2 \dx x \dx t\\
&=\iint_{\{(x,t)\in U^j : u(x,t)=v(x,t)\}} u^{m-1}|\nabla ((u-k_j)_-\zeta ) |^2 \dx x \dx t\\
&+ \iint_{(x,t)\in U^j :u(x,t)<v(x,t)\}} v^{m-1}|\nabla ((v-k_j)_-\zeta ) |^2 \dx x \dx t = I_1+I_2.\\
\end{align*}
Now $I_1$ can be estimated using Lemma \ref{caccioppoli} as
\begin{align*}
I_1&\le C \iint_{U^j}\Big( u^{m-1}|\nabla (u-k_j)_-|^2 \zeta^2 + u^{m-1}(u-k_j)_-^2 |\nabla \zeta|^2 \Big )\dx x \dx t \\
&\le C \iint_{U^j} \Big( (u-k_j)_-^2 \zeta |\zeta_t| + u^{m-1}(u-k_j)_-^2 |\nabla \zeta|^2 \Big)\dx x \dx t \\
&\le C\left ( \frac {k_j^2 2^{2(j+1)}}{\lambda \rho^2 } + \frac {2^{2(j+1)}}{k_j^{m+1}\rho^2} \right) |A_j|\\
&\le \frac {C (\xi \mu)^{m+1}2^{2j}}{\rho^2 } \left (1+ \frac 1 {\lambda (\xi\mu)^{m-1}} \right ) |A_j|.
\end{align*}
On the set $\{(x,t)\in U^j : u(x,t)<v(x,t)\}$, we have $v=\frac {a \xi \mu}{2} \le \xi \mu$. Therefore $I_2$ can be approximated by 
\begin{align*}
I_2&\le  (\xi \mu)^{m-1} \iint_{\{ (x,t)\in U^j :u(x,t)<v(x,t)\}}(v-k_j)_-^2 |\nabla \zeta|^2 \dx x \dx t  \\
&\le \frac {(\xi \mu)^{m-1} k_j^2 2^{2(j+1)}}{\rho^2} |A_j| \\
&\le \frac{C(\xi\mu)^{m+1}2^{2j}}{\rho^2} |A_j|. 
\end{align*}
Since $v\ge u$, we have $(u-k_j)_-\ge (v-k_j)_-$ and thus we may use Lemma \ref{caccioppoli} and the same reasoning as for the upper bound of $I_1$ to get, 
\[
\esssup_{t\in (T_j,T)} \int_{B^j}(v-k_j)_-^2 \zeta^2 \dx x \le \frac {C (\xi \mu)^{m+1}2^{2j}}{\rho^2 } \left (1+ \frac 1 {\lambda (\xi\mu)^{m-1}} \right ) |A_j|.
\]
Collecting the facts, \eqref{degiorgi Aj1} and \eqref{embedding} show that 
\begin{align*}
&\left(\frac {1-a}{2^{j+1}} \right )^2 (\xi \mu)^2 |A_{j+1}| \\
&\le \left( \frac{C 2^{2j}(\xi\mu)^{m+1}}{(\frac 12 a)^{m-1}\rho^2 (\xi \mu)^{m-1}}\left(1+\frac 1 {\lambda (\xi \mu)^{m-1}}\right)|A_j|\right)^{n/(n+2)}\\
&\times \left( \frac{C 2^{2j}(\xi\mu)^{m+1}} {(\frac 12 a)^{m-1} \rho^2 }\left(1+\frac {1}{\lambda(\xi \mu)^{m-1}}\right )|A_j|\right)^{2/(n+2)} |A_j|^{2/(n+2)} \\
&=\frac{C}{(\frac 12 a)^{m-1}} \frac {2^{2j}}{\rho^2} (\xi \mu)^{(2n+2(m+1))/(n+2)}\left(1+ \frac 1 {\lambda (\xi \mu)^{m-1}}\right) |A_j|^{1+2/(n+2)}.
\end{align*}
This can be written as
\[
|A_{j+1}|\le \frac{C}{(\frac 12 a)^{m-1}(1-a)^2} \frac {2^{4j}}{\rho^2}\left(\frac{\lambda(\xi\mu)^{m-1}+  1} {\lambda (\xi \mu)^{((m-1)n)/(n+2)}}\right) |A_j|^{1+2/(n+2)}.
\]
We denote $Y_j=\frac{|A_j|}{|U^j|}$. Next, We divide both sides by $\rho^{n+2}\lambda$ and observe, that $\rho^2 \lambda^{2/(n+2)}\ge C|U^j|^{2/(n+2)}$ to get
\[
Y_{j+1}\le 2^{4j}\frac{C}{(\frac 12 a)^{m-1}(1-a)^2}\left(\frac{\lambda(\xi\mu)^{m-1}+  1} {(\lambda (\xi \mu)^{m-1})^{n/(n+2)}}\right) Y_j^{1+2/(n+2)}.
\]
 Now by fast geometric convergence \cite[Lemma 7.1, p.220]{giusti}, $Y_j\rightarrow 0$, if 
\[
Y_0\le \left( \frac{ (\frac 1 2 a)^{m-1} (1-a)^2 }{C(\lambda (\xi\mu)^{m-1}+1)}\right)^{(n+2)/2} (\lambda (\xi \mu)^{m-1})^{n/2} 2^{-(n+2)^2}. 
\]
Choosing $\nu$ to be the quantity on the right hand side, this holds by the assumption on $u$. Thus $Y_j\rightarrow 0$ as $j\rightarrow \infty$, implying that $u\ge a \xi \mu$ almost everywhere in $B(x_0,\rho)\times (t_0+3\lambda \rho^2,t_0+4\lambda \rho^2)$. \end{proof}

Now, we have all the necessary tools for showing, that the expansion of positivity holds for $w$. We observe, that in the following lemma, we can make $\eps$ as small as we please by choosing $\theta$ larger. This, however, increases the waiting time. This turns out to be a useful property in the proof of the main theorem.

\begin{lemma}[Expansion of positivity for $w$]\label{w expansion}
Let $u$ be a non-negative weak supersolution in $\Omega_{T_0}$, such that 
\[
|\{x\in B(x_0,\rho): u(x,s)>k\} | \ge \gamma |B(x_0,\rho)|
\]
 at some time level $s\in(0,\rho^2)$ for some $\gamma\in(0,1)$ and let $w$ be defined as in Lemma \ref{change of variables}. Then there exists $\eps>0$, depending only on $m,n, \gamma$ and $\theta$, such that
  \[
w \ge \frac 1 2 \eps k_0 \text{ almost everywhere in } B(x_0,2\rho)
\]
for almost every
\[
 t\in\left( \frac{1+3\theta}4\frac {C}{(\eps k_0)^{m-1}}(2\rho)^2, \theta\frac {C}{(\eps k_0)^{m-1}}(2\rho)^2\right).
\]

\end{lemma}
\begin{proof}
  Let $a=\frac 12$ and suppose, that $\xi$ is chosen in such a way, that $\xi \mu = \eps k_0$. The claim holds by Lemma \ref{degiorgi}, if
  \begin{equation}
    \label{degiorgi condition}
|\{(x,t)\in U_{4\rho} : w<\eps k_0\} | \le \nu |U_{4\rho}|,
  \end{equation}
where the constants $t_0$ and $\lambda$ are chosen in such a way, that
\[
U_{4\rho}= B(x_0,4\rho)\times \left ( \frac{2}{(\eps k_0)^{m-1}}(4\rho)^2, \frac{2\theta}{(\eps k_0)^{m-1}}(4\rho)^2\right)
\]
and
\[
\nu=\left (\frac 1 {C4^m(2\theta-1)}\right)^{(n+2)/2} (2(\theta-1))^{n/2}2^{-(n+2)^2}.
\]
 By Lemma \ref{nu lemma}, we can choose $\eps$ in such a way, that \eqref{degiorgi condition} holds. We observe, that $\eps$ depends only on $m,n, \gamma$ and $\theta$. 

\end{proof}
Now, we will return to the original coordinates and show, that the expansion of positivity holds for $u$ as well. 

\begin{lemma}[Expansion of positivity for $u$]\label{u expansion}
  Let $u$ be a weak supersolution in $\Omega_{T_0}$, such that
\[
|\{x\in B(x_0,\rho) : u(x,s)>k\}| \ge \gamma |B(x_0,\rho)|.
\]
at some time level $s\in(0,\rho^2)$ for some $\gamma\in(0,1)$.
Then
\[
u\ge \eta k \quad \text {almost everywhere in } B(x_0,2\rho)\times \left(s+ \frac 12 \frac{b^{m-1}}{(k \eta)^{m-1}} \delta \rho^2 , s+\frac{b^{m-1}}{(k \eta)^{m-1}} \delta \rho^2 \right),
\]
here  
\begin{align*}
  &b=\frac {\eps \gamma}{16},\\
&\eta=\frac{b}{b_1^\theta}\quad \text{and }\\
&b_1=\exp \left( \frac{C}{(m-1)(\eps \gamma)^{m-1}\delta}\right).
\end{align*}
\end{lemma}

\begin{proof}
By Lemma \ref{w expansion}, 
  \[
w(\cdot,\tau) \ge \frac 1 2 \eps k_0 \text{ for almost every } \tau \in \left( \frac{1+3\theta}{4}\frac {C}{(\eps k_0)^{m-1}}(2\rho)^2, \theta\frac {C}{(\eps k_0)^{m-1}}(2\rho)^2\right).
\]
Recalling the definition of $k_0=\frac \gamma 8 (\delta \rho^2)^{\frac1{m-1}}$, this states
\[
\tau \in \left( \frac{1+3\theta}4\frac {C}{(\eps \gamma)^{m-1}\delta}, \theta\frac {C }{(\eps \gamma)^{m-1}\delta}\right)
\]
 and therefore
\begin{align*}
&e^{\frac\tau{m-1}} \in \left( \exp\left(\frac{1+3\theta}4\frac {C}{(m-1)(\eps \gamma)^{m-1}\delta}\right),\exp\left( \theta\frac {C}{(m-1)(\eps \gamma)^{m-1}\delta}\right)\right)\\
&=(b_1^{\frac{1+3\theta}4},b_1^\theta).
\end{align*}
Recalling the definition of $w$, we get the estimate
\[
w(x,\tau) = \frac{e^{\frac{\tau}{m-1}}}k (\delta \rho^2)^{\frac1 {m-1}} u(x,\Lambda^{-1}(\tau)) \le \frac {b_1^\theta }{k}(\delta \rho^2 )^{\frac 1 {m-1}} u(x,\Lambda^{-1}(\tau)),
\]  
where $\Lambda^{-1}(\tau)$ is defined as in Lemma \ref{change of variables}. Thus
\[
u(x,t)\ge \eta k \text{ for almost every } t\in \left( s+ \frac{b_1^{\frac{1+3\theta}{4}(m-1)}}{k^{m-1}} \delta \rho^2 , s+ \frac{b_1^{\theta (m-1)}}{k^{m-1}} \delta \rho^2 \right), 
\]
where $\eta = \frac{\eps \gamma}{16 b_1^\theta}$. Choosing $b=\frac{\eps \gamma}{16}$ concludes the proof.  
\end{proof}

Finally, we prove a refined version of Lemma \ref{u expansion}. The crucial feature in the following lemma is the power-like dependency of $\eta$ on the parameter $\gamma$, whereas in Lemma \ref{u expansion} the dependency is exponential. The idea of the proof is as follows. We use the measure theoretical lemma (Lemma \ref{measure theoretical lemma}) to find a small ball $B(\tilde x, \eps \rho)$, where $u$ is large in a fixed portion of the ball. Then we may use Lemma \ref{u expansion} iteratively to get the result. 

\begin{lemma}\label{expansion of positivity}
  Let $u\ge0$ be a weak supersolution, such that
\[
|\{x\in B(x_0,\rho) : u(x,s)>k\}| \ge \gamma |B(x_0,\rho)| 
\]
at some time level $s\in(0,\rho^2)$ for some $\gamma\in (0,1)$. Then, there exist constants $\eta_0, \delta \in (0,1)$, $b,d>1$ and a time level $\tilde t \in \left ( s + \frac 12 \frac{\delta \rho^2}{k^{m-1}}, s+ \frac{\delta \rho^2}{k^{m-1}}\right)$, such that 
\[
u(\cdot, t) \ge \eta k \quad \text{almost everywhere in } B(x_0,2\rho)
\]
for almost every $t\in \left ( \tilde t+ \frac 12\frac {b^{m-1}}{(\eta k)^{m-1}} \delta \rho^2, \tilde t+\frac {b^{m-1}}{(\eta k)^{m-1}} \delta \rho^2  \right)$. Here $\eta = \eta_0 \gamma^d$. 
 
\end{lemma}

\begin{proof}
Denote $T=\frac \delta {k^{m-1}} \rho^2$.   Let $U_\rho = B(x_0,\rho) \times \left (s+ \frac 12 T ,s+ T \right)$ and let $\widetilde{U}_\rho = B(x_0,2\rho) \times \left ( s ,s+ T \right)$. Choose a smooth cut-off function $\zeta$, such that $0\le \zeta \le 1$ and 
\[
\begin{cases}
  \zeta = 1 \quad \text{in } U_\rho,\\
\zeta = 0 \quad \text{on } \bd_p \widetilde{U}_\rho,\\
|\nabla \zeta | \le \frac C \rho \quad \text{and}\\
|\zeta_t| \le \frac {C} {T }. 
\end{cases}
\]
We observe, that
\begin{align}\label{(m+1)/2 estimate}
  &\iint_{U_\rho\cap \{u< k\}} |\nabla u^{\frac{m+1}{2}}|^2 \dx x \dx t =   \iint_{U_\rho\cap \{u< k\}} u^{m-1}|\nabla u|^2 \dx x \dx t \nonumber\\
&\le \iint_{\widetilde{U}_\rho} u^{m-1} |\nabla (u-k)_-|^2 \zeta^2 \dx x \dx t.\nonumber\\
\end{align}
By the Caccioppoli estimate in Lemma \ref{caccioppoli}, we have
\begin{align*}
 & \iint_{\widetilde{U}_\rho} u^{m-1} |\nabla (u-k)_-|^2 \zeta^2 \dx x \dx t \\
&\le C \left ( \iint_{\widetilde{U}_\rho} (u-k)_-^2 \zeta |\zeta_t| + u^{m-1}(u-k)_-^2 |\nabla \zeta|^2 \dx x \dx t \right)\\
&\le \frac {C k^{m+1} |\widetilde{U}_\rho|}{\delta \rho^2 }.
\end{align*}
Thus, H\"older's inequality, together with \eqref{(m+1)/2 estimate}, gives
\[
 \iint_{U_\rho\cap \{u< k\}} |\nabla u^{\frac{m+1}{2}}| \dx x \dx t \le \frac {C k^{\frac{m+1}2} |U_\rho|}{\gamma \rho } 
\]

Define a function
\[
w=\frac {(u^{\frac{m+1}2}-k^{\frac{m+1}2})_-}{k^{\frac{m+1}2}}.
\]
Now 
\[
\iint_{U_\rho} |\nabla w| \dx x \dx t = \frac 1 {k^{\frac{m+1}2}}  \iint_{U_\rho\cap \{u< k\}} |\nabla u^{\frac{m+1}{2}}| \dx x \dx t \le  \frac {C |U_\rho|}{\gamma \rho}.
\]
Thus, we have
\[
\frac 2 T\iint_{U_\rho} |\nabla w| \dx x \dx t \le \frac C \gamma \rho^{n-1}
\]
and therefore we can find $\tilde t \in \left(s+\frac 12 T,  s+T\right)$, such that 
\[
\int_{B(x_0,\rho)} |\nabla w(x,\tilde t) | \dx x \le \frac C \gamma \rho^{n-1}.
\]
On the other hand, by Lemma \ref{time propagation}, we have 
\[
\left |\left \{x\in B(x_0,\rho) : u(x,t)>\frac{\gamma}{8}k\right\}\right| \ge \frac \gamma 8 |B(x_0,\rho)|
\]
for almost every $t\in \left ( s, s+ \frac {\gamma^2 \rho^2 }{C k^{m-1}} \right ]$. We observe, that whenever $u> \frac {\gamma}{8}k$, we have
\[
w = \frac {( u^{\frac{m+1}2} - k^{\frac{m+1}2} )_-}{k^{\frac{m+1}2}} < 1 - \left ( \frac \gamma 8 \right)^{\frac{m+1}2}.
\]
Define a function
\[
v= \frac {1-w}{\left(\frac \gamma 8 \right)^{\frac{m+1}2}}.
\]
Now $v$ has the following properties:
\[
  |\{x\in B(x_0,\rho) : v(x,t ) > 1 \}| \ge \frac \gamma 8 |B(x_0,\rho)|
\]
for almost every  $t \in \left(s+\frac 12 T , s+T \right) $ and there exists $\tilde t \in \left (s+\frac 12 T , s+T \right)$, such that 
\[
\int_{B(x_0,\rho)} |\nabla v | \dx x = \frac 1 {\left( \frac \gamma 8\right)^{\frac {m+1}2}} \int_{B(x_0,\rho)} |\nabla w | \dx x \le \frac C {\gamma \left( \frac \gamma 8 \right)^{\frac{m+1}2}}\rho^{n-1}.
\]

By Lemma \ref{measure theoretical lemma} with constants $\delta=\frac 12$ and $\lambda=\frac 1 {2^{\frac{m+1}{2}}}$, we find a ball $B(\tilde x, \eps \rho)$, such that 
\[
\left |\left\{x\in B(\tilde x, \eps\rho) : v > \frac 1 {2^{\frac{m+1}2}} \right \} \right | > \frac 1 2 |B(\tilde x, \eps\rho)|.
\]
Here $\eps= C \left( \frac \gamma 8 \right)^{2+\frac{m+1}2} \gamma$. We observe, that whenever $v> \frac 1 {2^{\frac{m+1}2}}$, we have $w< 1- \left( \frac \gamma {16}\right)^{\frac {m+1}2}$ and thus $ u> \frac \gamma {16} k$. Therefore
\begin{equation}\label{exp of pos iteration}
\left |\left \{ x\in B(\tilde x, \eps \rho) : u(x,\tilde t) > \frac \gamma {16} k \right \} \right | \ge \frac 12 | B(\tilde x, \eps \rho)|
\end{equation}
at some time $\tilde t \in \left ( s+\frac 12 \frac {\delta \rho^2}{k^{m-1}}, s+\frac {\delta \rho^2}{k^{m-1}} \right).$ Denote
\[
T_i = \frac {\overline b^{m-1}}{(\overline \eta^i \tilde k )^{m-1}} \overline \delta (2^{i-1}\eps \rho)^2,
\]
where the constants $\overline b$, $\overline \eta$ and $\overline \delta$ correspond to $\gamma=\frac 12$ in Lemma \ref{u expansion} and $\tilde k=\frac \gamma {16}k$. Applying Lemma \ref{u expansion} to \eqref{exp of pos iteration} shows, that
\[
u(x,t)\ge \tilde \eta \tilde k \quad \text{almost everywhere in } B(\tilde x, 2\eps \rho)
\] 
for almost every $t_1\in \left ( \tilde t + \frac 12 T_1, \tilde t + T_1 \right)$. Applying Lemma \ref{u expansion} iteratively shows that
\[
u \ge \overline \eta^i \tilde k \quad \text{almost every where in } B(\tilde x, 2^i\eps \rho) 
\]
for almost every $t_i\in \left (t_{i-1}+\frac 12 T_i, t_{i-1}+T_i \right).$ Without loss of generality, we may assume $2^N \eps = 4$ for some $N\in \n$. Thus, we obtain 
\[
u \ge \overline \eta^N \tilde k \quad \text{almost everywhere in } B(\tilde x, 4\rho)  
\]
for almost every
\[
t\in \left ( \tilde t + \frac 12 \frac{\overline b^{m-1}}{\tilde k^{m-1}} \overline \delta (\eps \rho)^2 \sum_{i=1}^N \frac{4^{i-1}}{\overline \eta^{i(m-1)}}, \tilde t +  \frac{\overline b^{m-1}}{\tilde k^{m-1}} \overline \delta (\eps \rho)^2 \sum_{i=1}^N \frac{4^{i-1}}{\overline \eta^{i(m-1)}} \right ).
\]
This implies
\[
u \ge \overline \eta^N \tilde k \quad \text{almost everywhere in } B(x_0, 2\rho)  
\]
for almost every
\[
t\in \left ( \tilde t + \frac 23 \frac{\overline b^{m-1}}{(\tilde k \overline \eta^N)^{m-1}} \overline \delta (2\rho)^2, \tilde t +  \frac 56\frac{\overline b^{m-1}}{(\tilde k \overline \eta^N)^{m-1}} \overline \delta (2 \rho)^2\right ).
\]
Since $2^N \eps = 4$, we may write $N= 2+ \log_{\overline \eta} \eps^{- \frac {\ln \overline \eta}{\ln 2}}$. Recalling the definition of $\eps$ and $\tilde k$, we have 
\[
\overline \eta^N \tilde k = \eta_0 \gamma^d k,
\]
where $d= - \frac {\ln \overline \eta}{\ln 2 }\left ( 3+ \frac{m+1}{2} \right)+1$ and $\eta_0$ is a constant depending only on $m,n$ and $\theta$. Therefore, choosing suitable constants $b$ and $\delta$ gives
\[
u \ge \eta_0 \gamma^d k \quad \text{almost everywhere in } B(x_0,2\rho)
\]
for almost every
\[
t\in \left( \tilde t + \frac 12 \frac{b^{m-1}}{(\eta_0 \gamma^d k )^{m-1}} \delta \rho^2,  \tilde t + \frac{b^{m-1}}{(\eta_0 \gamma^d k )^{m-1}} \delta \rho^2 \right), 
\]
thus concluding the proof.
\end{proof}

\section{The cold alternative}
We will show, that if the supersolution $u$ is large only in a small portion of the ball $B(x_0,\rho)$ at every time level $t\in(0,\rho^2)$, then $u$ is bounded away from zero after some waiting time, provided that the integral average of $u$ over the ball is large enough at time $t=0$. The strategy of the proof is the following. We will use a qualitative version of a reverse H\"older's inequality to show, that the $L^q$-norms of supersolutions over cylinders of radius $\frac 34 \rho$ are uniformly bounded. Then, using the Caccioppoli estimates together with the boundedness of $L^q$-norms, we show that the $L^1$-norms of $\nabla u^m$ are bounded as well, thus giving us a uniform lower bound for the integral averages of $u$ over a smaller ball $B(x_0,\frac 58 \rho)$, provided that the integral average at $t=0$ is large enough. Finally, we use a real analytic lemma (Lemma \ref{simple measure lemma}) to find a time level $\tau\in\left(0,\left(\frac 58\rho\right)^2\right)$, such that 
\[
\left|\left\{x\in B\left(x_0,\frac 58\rho\right): u(x,\tau)>C_1\right\}\right| \ge C_2 \left|B\left(x_0,\frac 58\rho\right)\right|,
\]
 and thus the boundedness from below follows from Lemma \ref{expansion of positivity}.

First, we will prove a qualitative version of a reverse H\"older's inequality for the weak supersolutions of the porous medium equation.
\begin{lemma}\label{weak RHI}
Let $u$ be a weak supersolution in a neighbourhood of $B(x_0,\rho)\times(0,\rho^2)$, such that $u>0$. Let $q\in (m-1,m+\frac 2 n)$ and let $s$ be defined as 
\[
s=(m-1)+\left ( 1+ \frac 2 n \right )^{-(N+1)}(q-(m-1)),
\]
where $N\in \n$. If
\[
\fint_0^{\rho^2} \fint_{B(x_0,\rho)} u^s \dx x \dx t \le \widetilde{C},
\]
for some $\widetilde{C}$, then 
\[
\fint_0^{(\alpha \rho)^2} \fint_{B(x_0,\alpha \rho)} u^q \dx x \dx t \le C
\]
for every $\alpha\in \left(\frac 12,1\right)$. Here $C=C(m, n, q, N, \widetilde{C}, \alpha)$.  

\end{lemma}

\begin{proof}
  Fix $\alpha\in \left (\frac 12 , 1 \right )$ and $N\in \n$. Define 
\[
r_j = \rho - (1-\alpha)\rho\frac{1-2^{-j}}{1-2^{-(N+1)}}.
\]
Then $r_0=\rho$ and $r_{N+1}=\alpha \rho$. Denote $B^j=B(x_0,r_j)$ and $U^j=B^j\times (0,r_j^2)$. For fixed $j$, let $\zeta$ be a smooth cut-off function, such that $0\le\zeta \le 1$ and 
\begin{align*}
 & \zeta = 1 \quad\text{in } U^{j+1},\\
& \zeta = 0 \quad \text{on } \bd^p U^j,\\
&|\nabla \zeta| \le \frac C {r_j-r_{j+1}} \le \frac {C 2^{j+1}}{(1-\alpha)\rho} \quad \text{and}\\
& |\zeta_t| \le \frac C {r_j^2 - r_{j+1}^2} \le \frac {C 2^{2(j+1)}}{((1-\alpha)\rho)^2}.
\end{align*}
In order to utilize the Caccioppoli estimates, we choose 
\begin{align*}
&a=m-\eps,\\ 
&\kappa=1+ \frac {2(1-\eps)}{n(m-\eps)} \quad \text{and}\\
&b = 2 \frac{m-\eps}{1-\eps},
 \end{align*}
where $\eps\in (0,1)$, and use the parabolic Sobolev's inequality \cite[Proposition 3.1]{dibenedetto} with $q=2\kappa$, $p=2$ and $m=n(\kappa-1)$ to get the estimate
\begin{align*}
  &\fint_0^{r_{j+1}^2}\fint_{B^{j+1}} u^{\kappa a} \dx x \dx t = \fint_0^{r_{j+1}^2}\fint_{B^{j+1}} \big(u^{a/2}\zeta^{b/2}\big)^{2\kappa} \dx x \dx t \\
& = \frac C {r_{j+1}^{n+2}} \iint_{U^{j+1}} \big(u^{a/2}\zeta^{b/2}\big)^{2\kappa} \dx x \dx t 
\le \frac {C 2^{n+2}} {r_{j}^{n+2}} \iint_{U^j}  \big(u^{a/2}\zeta^{b/2} \big)^{2\kappa} \dx x \dx t\\
&\le \frac {C 2^{n+2}} {r_{j}^{n+2}} \iint_{U^j} |\nabla \big(u^{a/2}\zeta^{b/2} \big) |^2 \dx x \dx t \left ( \esssup_{t\in(0,r_j^2)} \int_{B^j} \big(u^{a/2}\zeta^{b/2} \big)^{n(\kappa-1)} \dx x \right)^{2/n}.
\end{align*}
In the previous inequality, we may bypass the boundedness assumption in the parabolic Sobolev's inequality by considering $\min\{u,k\}$ and using the monotone convergence theorem to pass to the limit. By the choice of $a$ and $b$, we get the estimate
\begin{align*}
&|\nabla \big(u^{a/2}\zeta^{b/2} \big) |^2  \\
&\le C \left( \left(\frac{m-\eps}{2}\right)^2 u^{m-\eps-2} \zeta^{2 \left( \frac{m-\eps}{1-\eps}\right)}|\nabla u|^2 + \left(\frac{m-\eps}{1-\eps}\right)^2 u^{m-\eps} |\nabla \zeta|^2 \right )\\
&\le C \left( u^{m-\eps-2} \zeta^2 |\nabla u|^2 + \frac{1}{(1-\eps)^2} u^{m-\eps} |\nabla \zeta|^2 \right ).
\end{align*}
Here we used the fact that $\eps \in (0,1)$. Thus, by Lemma \ref{caccioppoli2}, we get
\begin{align*}
 & \iint_{U^j} |\nabla \big(u^{a/2}\zeta^{b/2} \big) |^2 \dx x \dx t \\
&\le C  \iint_{U^j} \left ( u^{m-\eps-2}\zeta^2 |\nabla u|^2  +  \frac 1 {(1-\eps)^2} u^{m-\eps} |\nabla \zeta|^2 \right) \dx x \dx t \\
&\le \frac C {|\eps|^2 (1-\eps)^2 } \left (\iint_{U^j} u^{m-\eps} |\nabla \zeta|^2 \dx x \dx t + \iint_{U^j} u^{1-\eps}  |\zeta_t| \dx x \dx t \right).
\end{align*}
Again, by the choice of $a$, $b$ and $\kappa$, we have $u^{\frac a 2 n(\kappa-1)}\zeta^{\frac b 2 n(\kappa-1)}= u^{1-\eps}\zeta^2$ and thus by Lemma \ref{caccioppoli2}, we get
\begin{align*}
 & \left( \esssup_{t\in (0,r_j^2) } \int_{B^j} \big(u^{a/2}\zeta^{b/2} \big)^{n(\kappa-1)} \dx x \right)^{2/n} \\
&\le \left(\frac C {|\eps|^2 (1-\eps)^2 } \left (\iint_{U^j} u^{m-\eps} |\nabla \zeta|^2 \dx x \dx t + \iint_{U^j} u^{1-\eps}  |\zeta_t| \dx x \dx t \right)\right)^{2/n}.
\end{align*}
So far we have 
\begin{align*}
  &\fint_0^{r_{j+1}^2}\fint_{B^{j+1}} u^{\kappa a} \dx x \dx t \le \\
&\frac {C}{r_j^{n+2}} \left(\frac 1{ \eps^2 (1-\eps)^2 }  \iint_{U^j} u^{m-\eps} |\nabla \zeta|^2 \dx x \dx t + \iint_{U^j} u^{1-\eps} |\zeta_t| \dx x \dx t \right )^{1+2/n} \le\\
&\left ( \frac C {\eps^2 (1-\eps)^2} \frac {2^{2(j+1)}}{(1-\alpha)^2} \left( \fint_{0}^{r_j^2}\fint_{B^j}  u^{m-\eps} \dx x \dx t  +  \fint_{0}^{r_j^2}\fint_{B^j} u^{1-\eps} \dx x \dx t \right )\right) ^{1+2/n} \\
\end{align*}
We denote $\sigma=\frac{1-\eps}{m-\eps}$. Now, H\"older's inequality gives us 
\begin{align}\label{almost iteration}
  &\fint_0^{r_{j+1}^2}\fint_{B^{j+1}} u^{m+\frac 2 n - \left (1+ \frac 2n \right)\eps} \dx x \dx t \le \nonumber\\
& \left ( \frac C {\eps^2 (1-\eps)^2} \frac {2^{2(j+1)}}{(1-\alpha)^2} \left( \fint_{0}^{r_j^2}\fint_{B^j}  u^{m-\eps} \dx x \dx t  +  \left (\fint_{0}^{r_j^2}\fint_{B^j} u^{m-\eps} \dx x \dx t \right)^\sigma\right )\right) ^{1+2/n}. \nonumber\\
\end{align}
Next, we denote $\gamma=1+\frac 2 n$ and $\eps_0=1-\gamma^{-(N+1)}(q-(m-1))$. Let $\eps_j = 1-\gamma^j(1-\eps_0)$ and $\delta_j=m-\eps_j$. Now we have
\begin{align*}
  &\delta_{N+1}= m-\eps_{N+1}= m+ \gamma^{N+1}(\gamma^{-(N+1)}(q-(m-1))-1 = q \quad \text{and}\\
&\delta_0 = m-\eps_0 = m-1+ \left(1+\frac 2n\right)^{-(N+1)}(q-(m-1))=s.
\end{align*}
We observe
\begin{align*}
  &m+ \frac 2n - \gamma \eps_j = m+ \frac 2n + \gamma^{j+1}(1-\eps_0)-\gamma = m  + \gamma^{j+1}(1-\eps_0)-1 \\
&=m - \eps_{j+1}=\delta_{j+1}.
\end{align*}
Thus, denoting 
\[
\Lambda_j=\fint_{0}^{r_j^2} \fint_{B^j} u^{\delta_j} \dx x \dx t,
\]
 \eqref{almost iteration} can be written as
\[
\Lambda_{j+1} \le \left ( \frac {C 2^{2j}}{\eps_j^2(1-\eps_j)^2 (1-\alpha)^2}(\Lambda_j +  \Lambda_j^\sigma) \right )^\gamma.
\]
In order to estimate the term $\frac{1}{\eps_j^2(1-\eps_j)^2}$, we want to show that $\eps_j$ is an decreasing sequence, and that $\eps_N>0$. Consider the difference 
\[
\eps_{j+1}-\eps_j = (1-\eps_0)\gamma^j(1-\gamma) = (\eps_0-1)\gamma^j\frac 2 n.
\]
By assumption $q>(m-1)$, implying that $\eps_0<1$. Hence the sequence $\eps_j$ is decreasing. On the other hand
\[
\eps_N = 1-\gamma^{-1}(q-(m-1))>0,
\]
because $q< m+\frac 2 n = m-1+\gamma$, by assumption. We now have the estimate
\[
\frac{1}{\eps_j^2(1-\eps_j)^2} \le \frac{1}{\eps_N^2(1-\eps_0)^2}= \frac {\gamma^2}{(m+\frac 2n -q)^2 (s-(m-1))^2}.
\]
We denote $\lambda= \frac{1}{(m+\frac 2n -q) (s-(m-1))}$. Thus 
\begin{equation}\label{iteration}
 \Lambda_{j+1} \le \left ( \frac {C 2^{2j}\lambda^2}{(1-\alpha)^2}(\Lambda_j+\Lambda_j^\sigma )\right )^\gamma.
\end{equation}
By assumption
\[
\Lambda_0=\fint_0^{\rho^2} \fint_{B(x_0,\rho)} u^s \dx x \dx t \le \widetilde{C}.
\]
Thus, iterating \eqref{iteration} $N+1$ times gives
\[
\fint_0^{(\alpha\rho)^2} \fint_{B(x_0,\alpha\rho)} u^q \dx x \dx t= \Lambda_{N+1} \le C.
\]
\end{proof}

Next, we will show, that the $L^q$-norms of these supersolutions are uniformly bounded.

\begin{lemma}\label{uniform Lq bound}
  Let $u$ be a weak supersolution in $\Omega_{T_0}$, such that
\[
|\{x\in B(x_0,\rho): u(x,t)>k^{1+\frac 1d}\}| \le  k^{-\frac 1d} |B(x_0,\rho)|,
\]
for every $k>1$ and for almost every $t\in (0,\rho^2).$ Here $d$ is as in Lemma \ref{expansion of positivity}. Then for $q\in \left(m-1,m+\frac 2 n\right)$ we have
\[
\fint_0^{\left(\frac 34\rho\right)^2} \fint_{B\left(x_0,\frac 34\rho\right)} u^q \dx x \dx t \le C.
\]
\end{lemma}
\begin{proof}
  Let $\delta=\frac{1}{2d+2}$ and let $t\in(0,\rho^2)$. Then, by applying Cavalieri's principle at the time level $t$, we have
  \begin{align*}
    &\int_{B(x_0,\rho)} u^\delta \dx x = \delta \int_0^\infty \lambda^{\delta-1} |\{x\in B(x_0,\rho): u(x,t)>\lambda\}| \dx \lambda \\
&= C \int_0^\infty k^{\left(1+\frac 1d\right)(\delta-1)}|\{x\in B(x_0,\rho): u(x,t)>k^{1+\frac 1d}\}| k^{\frac 1d} \dx k\\
&\le C\left(|B(x_0,\rho)|+ \int_1^\infty k^{\left(1+\frac 1d\right)(\delta-1)}|\{x\in B(x_0,\rho): u(x,t)>k^{1+\frac 1d}\}| k^{\frac 1d} \dx k \right)\\
&\le C |B(x_0,\rho)| \left( 1+\int_1^\infty k^{-\left (1+\frac 1{2d}\right)} \dx k \right) \le C |B(x_0,\rho)|.
  \end{align*}
Here we used the assumption $|\{x\in B(x_0,\rho): u(x,t)>k^{1+\frac 1d}\}| k^{\frac 1d}\le |B(x_0,\rho)|$. Thus 
\begin{equation}\label{udelta}
\fint_{B(x_0,\rho)} u^\delta \dx x \le C
\end{equation}
for almost every $t\in(0,\rho^2)$. Denote $U(s)=B(x_0,s)\times (0,s^2),$ for $s\in \left ( \frac 78 \rho,\rho\right)$. Let $\frac 78 \rho\le s < S\le  \rho$ and take a smooth cut-off function $\zeta \in C_0^\infty(B(0,S))$, such that $0\le \zeta \le 1$ and
\begin{align*}
 & \zeta=1 \quad \text{in } U(s) \quad \text{and}\\
&|\nabla \zeta| \le \frac C {S-s}.
\end{align*}
As in the proof of Lemma \ref{weak RHI}, we want to use the parabolic Sobolev's inequality and Caccioppoli estimates. We choose 
\begin{align*}
  &a=m-1+\delta,\\
&\kappa=1+\frac{2\delta}{n(m-1+\delta)} \quad \text{and }\\
&b=2
\end{align*}
and thus we obtain
\begin{align*}
&\iint_{U(s)} u^{\kappa a} \dx x \dx t\\
& \le \iint_{U(S)} |\nabla (u^{\frac a 2}\zeta^{\frac b 2 } )|^2 \dx x \dx t \left ( \esssup_{t\in (0,S^2)} \int_{B(x_0,S)}|u^{\frac a 2} \zeta^{\frac b 2 }|^{(\kappa-1)n}\dx x \right)^{2/n}.
\end{align*}
We observe $\kappa a= m-1+\delta(1+\frac 2 n)$ and $|u^{\frac a 2} \zeta^{\frac b 2 }|^{(\kappa-1)n} = u^\delta \zeta^{\frac{2\delta}{m-1+\delta}}$. Moreover, we may estimate
\[
|\nabla (u^{\frac a 2}\zeta^{\frac b 2 } )|^2 \le C(u^{m-3+\delta} \zeta^2 |\nabla u|^2 + u^{m-1+\delta}|\nabla \zeta|^2 ).
\]
Combining these estimates, we have
\begin{align*}
  \iint_{U(s)} u^{m-1+\delta(1+\frac 2 n)} \dx x \dx t
&\le C \iint_{U(S)} \Big(u^{m-3+\delta}\zeta^2 |\nabla u|^2 + u^{m-1+\delta} |\nabla \zeta|^2 \Big) \dx x \dx t \\ &\times\left ( \esssup_{t\in(0,S^2)} \int_{B(x_0,S)} u^\delta \zeta^{\frac{2\delta} {m-1+\delta}} \dx x \right)^{2/n}.
\end{align*}
Since $\zeta\le 1$, we may use \eqref{udelta} to get
\[
\esssup_{t\in(0,S^2)} \int_{B(x_0,S)} u^\delta \zeta^{\frac{2\delta} {m-1+\delta}} \dx x \le \esssup_{t\in(0,\rho^2)} \int_{B(x_0,\rho)} u^\delta  \dx x \le C \rho^n.
\]
We may set $\eps=1-\delta$ and use Lemma \ref{caccioppoli2} to get the estimate
\begin{align*}
 & \iint_{U(S)}u^{m-\eps-2}\zeta^2|\nabla u|^2 \dx x \dx t + \iint_{U(S)} u^{m-\eps}|\nabla \zeta|^2 \dx x \dx t \\
&\le C\iint_{U(S)} u^{m-\eps} |\nabla \zeta|^2 \dx x \dx t+ \int_{B(x_0,S)} u^{1-\eps}(x,S^2) \zeta(x)^2 \dx x\\
&\le C\iint_{U(S)} u^{m-\eps} |\nabla \zeta|^2 \dx x \dx t + C \rho^n.
\end{align*}
Here we used the fact that $\zeta_t=0$ and \eqref{udelta}. We have
\[
 \fint_{0}^{s^2} \fint_{B(x_0,s)} u^{m-1+\delta\left(1+\frac 2 n\right)} \dx x \dx t \le C \rho^2\fint_{0}^{S^2} \fint_{B(x_0,S)} u^{m-1+\delta} |\nabla \zeta|^2 \dx x \dx t+C.
\]
Next, we will use Young's inequality with 
\begin{align*}
&  p=\frac{m-1+\delta\left(1+\frac 2 n \right)}{m-1+\delta} \quad \text{and}\\
& q=\frac {n(m-1+\delta\left(1+\frac 2 n\right)}{2\delta}.
\end{align*}
to get
\begin{align*}
 & C\rho^2\fint_{0}^{S^2} \fint_{B(x_0,S)} u^{m-1+\delta} |\nabla \zeta|^2 \dx x \dx t \\
&\le C_1 \fint_{0}^{S^2} \fint_{B(x_0,S)} u^{(m-1+\delta)p}\dx x \dx t + C_2 \fint_{0}^{S^2} \fint_{B(x_0,S)} \rho^{2q}|\nabla \zeta|^{2q}\dx x \dx t \\
&\le C_1 \fint_{0}^{S^2} \fint_{B(x_0,S)} u^{m-1+\delta\left(1+\frac 2 n\right)}\dx x \dx t + C \left(\frac \rho {S-s}\right)^{n(m-1+\delta(1+2/n))/\delta}. 
\end{align*}
We choose the constants in such a way, that $C_1<1$ and thus $C$ is determined accordingly. We denote 
\begin{align*}
&\phi(s)=\fint_{0}^{s^2} \fint_{B(x_0,s)} u^{m-1+\delta\left(1+\frac 2 n\right)} \dx x \dx t \quad \text{and}\\
&\sigma={n(m-1+\delta(1+2/n))/\delta}.
\end{align*}
Thus we have 
\[
\phi(s)\le C\rho^\sigma(S-s)^{-\sigma} + C_1 \phi(S).
\]
Now by \cite[Lemma 8.15]{giaquinta}
\[
\fint_{0}^{s^2} \fint_{B(x_0,s)} u^{m-1+\delta(1+\frac 2 n)} \dx x \dx t  \le C\left(\frac \rho {S-s} \right)^{\sigma}.
\]
Choosing $s=\frac 78 \rho$ and $S=\rho$ shows that
\[
\fint_0^{\left(\frac 78 \rho \right )^2} \fint_{B\left( x_0, \frac 78 \rho \right)} u^{m-1+\delta(1+\frac 2 n)} \dx x \dx t  \le C.
\]
 Let $q\in \left( m-1, m+\frac 2 n \right)$. Note, that the constant $d$ in Lemma \ref{expansion of positivity} can be chosen to be as large as we please. Thus we may choose $d$ in such a way, that 
\[
\delta=\left(1+\frac 2n\right)^{-(N+1)}(q-(m-1)), \quad \text{for some } N.
\]
Now, applying Lemma \ref{weak RHI} in $B\left( x_0, \frac 78 \rho \right)\times \left (0, \left(\frac 78 \rho\right)^2\right)$, with $\alpha=\frac 67$ concludes the proof. 
\end{proof}
Next we will show the boundedness of the $L^1$-norms of the gradients $\nabla u^m$.

\begin{lemma}\label{uniform gradient bound}
  Let $u$ be a weak supersolution in $\Omega_{T_0}$, such that $u>0$. Suppose, that
\[
|\{x\in B(x_0,\rho): u(x,t)>k^{1+\frac 1d}\}| \le  k^{-\frac 1d} |B(x_0,\rho)|,
\]
for every $k>1$ and for almost every $t\in (0,\rho^2).$ Then
\[
\fint_0^{\left( \frac 58 \rho\right)^2} \fint_{B\left (x_0,\frac 58\rho \right)} |\nabla u^m| \dx x \dx t \le \frac C \rho.
\]  
\end{lemma}

\begin{proof}
Denote $U^1=B\left( x_0, \frac 58 \rho \right)\times \left (0, \left(\frac 58 \rho\right)^2\right)$ and $U^2=B\left( x_0, \frac 78 \rho \right)\times \left (0, \left(\frac 78 \rho\right)^2\right)$. Take $\eps \in \left( 0, \frac 2 n \right)$ and use H\"older's inequality to get
\begin{align*}
&\frac{1}{|U^1|}\iint_{U^1} |\nabla u^m| \dx x \dx t = \frac{1}{|U^1|}\iint_{U^1} m u^{m-1}|\nabla u| \dx x \dx t \\
&= \frac{1}{|U^1|}\iint_{U^1} m u^{\frac{m-\eps-2}{2}}|\nabla u|u^{\frac{m+\eps}{2}} \dx x \dx t\\
&\le \left(\frac{1}{|U^1|}\iint_{U^1} m u^{m-\eps-2}|\nabla u|^2 \dx x \dx t \right)^{1/2} \left(\frac{1}{|U^1|}\iint_{U^1} u^{m+\eps} \dx x \dx t \right)^{1/2}. 
\end{align*}
Let $\zeta$ be a smooth cut-off function, such that $0\le \zeta \le 1$ and 
\begin{align*}
 & \zeta=1 \quad \text{in } U^1, \\
& \zeta=0 \quad \text{on } \bd^p U^2 ,\\
& |\nabla \zeta| \le \frac C \rho \quad \text{and }\\
&|\zeta_t|\le \frac C {\rho^2}.
\end{align*}
We may use Lemma \ref{caccioppoli2} to control the first term on the right hand side 
\begin{align*}
 &\frac{1}{|U^1|} \iint_{U^1} m u^{m-\eps-2}|\nabla u|^2 \dx x \dx t \\
&\le \frac C {\eps^2(1-\eps)}\left ( \frac{1}{|U^2|}\iint_{U^2} u^{m-\eps} |\nabla \zeta|^2 \dx x \dx t + \frac{1}{|U^2|}\iint_{U^2} u^{1-\eps} \zeta |\zeta_t| \dx x \dx t \right)\\
&\le \frac C {\eps^2(1-\eps) \rho^2}\left ( \frac{1}{|U^2|}\iint_{U^2} u^{m-\eps} \dx x \dx t + \frac{1}{|U^2|}\iint_{U^2} u^{1-\eps} \dx x \dx t \right)\\
&\le \frac C {\eps^2(1-\eps)\rho^2} \frac{1}{|U^2|}\iint_{U^2} u^{m-\eps} \dx x \dx t \\
&+\frac C {\eps^2(1-\eps)\rho^2} \left(\frac{1}{|U^2|}\iint_{U^2} u^{m-\eps} \dx x \dx t\right)^{(1-\eps)/(m-1)}.
\end{align*}
We may assume, that $\eps$ is chosen in such a way, that we have $m-1<m-\eps<m+\eps<m+\frac 2 n$ and thus we may use Lemma \ref{uniform Lq bound} to conclude
\begin{align*}
&\fint_{0}^{\left(\frac 58 \rho\right)^2}\fint_{B\left(x_0,\frac 58\rho\right)} |\nabla u^m| \dx x \dx t \\
&\le\frac C \rho \left( \frac{1}{|U^2|}\iint_{U^2} u^{m-\eps} \dx x \dx t+ \left(\frac 1 {|U^2|} \iint_{U^2} u^{m-\eps} \dx x \dx t \right)^{(1-\eps)/(m-1)}\right)^{1/2}\\
&\times\left(\frac 1{|U^1|} \iint_{U^1} u^{m+\eps} \dx x \dx t \right)^{1/2} \\
&\le \frac C \rho.
\end{align*}
\end{proof}
Now, we will show, that the previous lemma gives us a uniform lower bound for the integral averages of $u$ for almost every time level $t\in\left(0,\left(\frac 58\rho \right)^2\right)$, if the integral average at $t=0$ is large enough.
\begin{lemma}\label{average essinf lemma}
  Let $u$ be a non-negative weak supersolution in $\Omega_{T_0}$, such that
\[
|\{x\in B(x_0,\rho): u(x,t)>k^{1+\frac 1d}\}| \le  k^{-\frac 1d} |B(x_0,\rho)|,
\]
for every $k>1$ and for almost every $t\in (0,\rho^2).$ There exists $C>0$, such that if
\[
\fint_{B\left(x_0,\frac 12 \rho\right)} u(x,0) \dx x \ge 2C,
\]
then
\[
\essinf_{t\in \left (0,\left (\frac 58 \rho\right)^2\right )} \fint_{B\left (x_0,\frac 58\rho\right)} u(x,t) \dx x\ge C.
\]
\end{lemma}

\begin{proof}
  Let $\zeta \in C_0^\infty\left(B\left(x_0,\frac 58\rho\right)\right)$ be a cut-off function, such that $0\le \zeta\le 1$ and 
  \begin{align*}
 & \zeta=1\quad \text{in } B\left(x_0,\frac 12\rho\right) \quad \text{and}\\
&|\nabla \zeta|\le \frac{\widetilde{C}}{\rho}.  
  \end{align*}
By \eqref{integration by parts in time} we have
\begin{align*}
  &\fint_{B\left(x_0,\frac 58\rho\right)} u(x,\tau) \zeta(x) \dx x \\
&\ge   \fint_{B\left (x_0,\frac 58\rho\right)} u(x,0) \zeta(x) \dx x - \int_0^\tau  \fint_{B\left (x_0,\frac 58\rho\right)} |\nabla u^m \cdot \nabla \zeta| \dx x \dx t
\end{align*}
for almost every $\tau\in \left(0,\left(\frac 58\rho\right)^2\right)$. Using Lemma \ref{uniform gradient bound}, we may estimate
\[
\int_0^\tau  \fint_{B\left(x_0,\frac 58\rho\right)} |\nabla u^m \cdot \nabla \zeta |\dx x \dx t\le \frac{\widetilde{C}}{\rho} \int_0^\tau  \fint_{B\left(x_0,\frac 58\rho\right)} |\nabla u^m | \dx x \dx t \le C. 
\]
Thus, we have
\begin{align*}
&  \fint_{B\left(x_0,\frac 58\rho\right)} u(x,\tau) \zeta(x) \dx x  \\
&\ge \fint_{B\left(x_0,\frac 12\rho\right)} u(x,0) \dx x - \int_0^\tau  \fint_{B\left(x_0,\frac 58\rho\right)} |\nabla u^m \cdot \nabla \zeta| \dx x \dx t\\
&\ge C
\end{align*}
for almost every $\tau\in \left(0,\left(\frac 58\rho\right)^2\right)$. We conclude
\[
\essinf_{t\in \left(0,\left (\frac 58\rho\right)^2\right)} \fint_{B\left(x_0,\frac 58\rho\right)} u(x,t)\ge C.
\]

\end{proof}

We will prove the following simple lemma for the readers convenience. 

\begin{lemma}\label{simple measure lemma}
  Let $\Omega \subset \rn$ be a bounded domain and let $f$ be a measurable function in $\Omega$. Suppose that 
  \begin{align*}
    \fint_\Omega f \dx x \ge 2C \quad \text{and} \quad \left( \fint_\Omega f^q \dx x \right)^{1/q} \le \lambda C,
  \end{align*}
for some $\lambda \ge 2$ and $q\in (1,\infty]$. Then
\[
|\{x\in \Omega : f(x)> C \} | \ge \lambda^{\frac {-q}{q-1}} |\Omega|. 
\]
\end{lemma}

\begin{proof}
  We have
  \begin{align*}
  2C &\le \fint_\Omega f \dx x = \frac 1 {|\Omega|} \left ( \int_{\{f>C\}} f \dx x + \int_{\{f\le C\}}  f \dx x \right)\\
&\le \frac 1 {|\Omega|}  \int_{\{f>C\}} f \dx x+C.
  \end{align*}
Denote $\Omega_C=\{x\in \Omega: f(x)>C\}.$ We use H\"older's inequality to control the first term on the right hand side by
\begin{align*}
  \frac 1 {|\Omega|}  \int_{\Omega_C} f \dx x &\le \frac 1 {|\Omega|}\left ( \int_{\Omega_C} f^q \dx x \right)^{1/q} |\Omega_C|^{(q-1)/q} \\
& \le \left ( \fint_\Omega f^q \dx x \right)^{1/q} \left( \frac {|\Omega_C|}{|\Omega|}\right)^{(q-1)/q} \\
&\le \lambda C \left( \frac {|\Omega_C|}{|\Omega|}\right)^{(q-1)/q}.
\end{align*}
Thus we have
\[
2C \le \lambda C  \left( \frac {|\Omega_C|}{|\Omega|}\right)^{(q-1)/q} + C,
\]
which implies 
\[
|\{x\in \Omega : f(x)>C\}| = |\Omega_C|\ge \lambda^{\frac{-q}{q-1}}|\Omega|,
\]
concluding the proof.

\end{proof}

Finally, we collect the results and apply Lemma \ref{simple measure lemma} to show that we'll end up in a situation, where we can apply the expansion of positivity.

\begin{lemma}\label{cold alternative}
  Let $d>1$ be as in Lemma \ref{expansion of positivity} and let $u$ be a weak supersolution in $\Omega_{T_0}$ , such that 
\[
|\{x\in B(x_0,\rho): u(x,t)>k^{1+\frac 1d}\}| \le  k^{-\frac 1d} |B(x_0,\rho)|
\]
for every $k>1$ and for almost every $t\in (0,\rho^2)$. There exists $M>0$, such that if
\[
\fint_{B\left(x_0,\frac 12\rho\right)} u(x,0) \dx x \ge M,
\]
then there exists $\tau \in \left(0,\left(\frac 58\rho \right)^2\right)$, such that 
\[
\left|\left\{x\in B\left (x_0,\frac 58\rho\right) : u(x,\tau) > C_1\right\}\right| \ge C_2 \left |B\left (x_0, \frac 58\rho\right)\right|,
\]
for some $C_1,C_2>0$ depending only on $m$ and $n$.
\end{lemma}

\begin{proof}
Let $C$ be the constant given by Lemma \ref{average essinf lemma}. For $M$ large enough, we have 
\[
\fint_{B\left(x_0,\frac 12\rho\right)} u(x,0) \dx x \ge 2C,
\]
and thus we may use Lemma \ref{average essinf lemma} to get
\[
\essinf_{t\in\left(0,\left(\frac 58\rho \right)^2\right)} \fint_{B\left(x_0,\frac 58\rho\right)} u(x,t)\dx x \ge C.
\]
Take $q>1$. By Lemma \ref{uniform Lq bound}, we have
\[
\fint_0^{\left(\frac 58\rho\right)^2} \fint_{B\left(x_0,\frac 58\rho\right)} u^q \dx x \dx t \le \widetilde C \fint_0^{\left(\frac 34\rho\right)^2} \fint_{B\left(x_0,\frac 34\rho\right)} u^q \dx x \dx t \le \widetilde C.
\]
We may choose $\tau\in \left(0, \left(\frac 58\rho\right)^2\right)$ and $\lambda\ge 2$, such that 
\[
\fint_{B\left(x_0,\frac 58\rho\right)} u(x,\tau)^q \dx x \le \lambda C. 
\]
Now, applying Lemma \ref{simple measure lemma} gives 
\[
\left |\left \{x\in B\left (x_0, \frac 58\rho\right): u(x,\tau)> C\right\}\right | \ge \lambda^{\frac{-q}{q-1}}\left|B\left(x_0,\frac 58\rho\right)\right|.
\]
Therefore, the Lemma holds for $C_1=C$ and $C_2=\lambda^{\frac{-q}{q-1}}$. 
\end{proof}

\section{Proof of Theorem \ref{main theorem}}
We are now ready to prove the main theorem. The idea of the proof is the following. Either $u$ is large in a large portion of the ball $B(x_0,\rho)$ at some time level $s$ or this does not happen at any time level. In the first case we may apply Lemma \ref{expansion of positivity} to show, that $u$ is essentially bounded from below by $\eta>0$. In the latter case we utilize Lemma \ref{cold alternative} to end up in a situation, where Lemma \ref{expansion of positivity} can be applied. After this, the conclusion follows from a scaling argument.  
\begin{proof}
  [Proof of Theorem \ref{main theorem}]
Denote 
\[
N=\fint_{B(x_0,\rho)}u(x,t_0) \dx x.
\]
We may assume $N>0$. Define a scaled function
\[
v(x,t)=\frac M N u\left(2x,t_0+\left (\frac M N \right)^{m-1}t \right),
\]
where $M$ is as in Lemma \ref{cold alternative}. Since $u$ is a weak supersolution in $\Omega_T$, $v$ is a weak supersolution in $\tilde\Omega\times \left( 0, \left(\frac N M \right)^{m-1}(T_0-t_0)\right)$. Moreover, $v$ has the property
\[
\fint_{B\left(x_0,\frac 12\rho\right)} v(x,0)\dx x = M.
\]
One of the following alternatives holds. Either, there exists a time level $s\in(0,\rho^2)$ and a constant $k>1$, such that 
\begin{equation}\label{hot alternative}
|\{x\in B(x_0,\rho): v(x,s)> k^{1+\frac 1d} \}| \ge \frac 1 {k^{\frac 1d}} |B(x_0,\rho)|
\end{equation}
or this does not hold for any pair $s$ and $k$. In the spirit of \cite{kuusi}, we call the former alternative ``hot'' and the latter ``cold''. We will first consider the hot alternative. Suppose that there exists $s\in(0,\rho^2)$ and $k>1$, such that \eqref{hot alternative} holds. Then, in particular, we have
\[
|\{x\in B(x_0,\rho): v(x,s)> k\}| \ge \frac 1 {k^{\frac 1d}} |B(x_0,\rho)|
\]
and thus Lemma \ref{expansion of positivity} implies
\begin{equation}\label{hot lower bound}
v\ge \eta_0
\end{equation}
almost everywhere in $B(x_0,2\rho)$ for almost every
\[
t\in\left( \tilde t + \frac 12  \frac {b^{m-1}}{\eta_0^{m-1}} \delta \rho^2 , \tilde t + \frac {b^{m-1}}{\eta_0^{m-1}} \delta \rho^2 \right),
\]
where $\tilde t \in \left( s+ \frac 12 \frac \delta {k^{m-1}} \rho^2,s+  \frac \delta {k^{m-1}} \rho^2 \right).$ We approximate $\tilde t < 2\rho^2$. Since $\eta_0$ can be chosen to be as small as we please, we may assume, that \eqref{hot lower bound} holds for almost every $t\in \left(2\rho^2 + \frac 12 T_h,T_h \right)$, where $T_h= \frac{b^{m-1}}{\eta_0^{m-1}} \delta \rho^2$. 

Next, we deal with the cold alternative. Suppose, that \eqref{hot alternative} does not hold for any pair $s,k$. Then, by Lemma \ref{cold alternative}, there exist $C_1,C_2>0$, such that
\[
\left|\left\{x\in B\left (x_0,\frac 58\rho\right) : v(x,t) > C_1\right\}\right| \ge C_2 \left |B\left (x_0, \frac 58\rho\right)\right|.
\]
Applying Lemma \ref{expansion of positivity} twice and doing a similar approximation as in the hot alternative shows that there exist $\eta_1>0$ and $T_c=\frac{b^{m-1}}{\eta_1^{m-1}} \delta \rho^2$, such that
\begin{equation}
  \label{cold lower bound}
v\ge \eta_1
\end{equation}
almost everywhere in $B(x_0,2\rho)$ for almost every $t\in(2\rho^2 + \frac 12 T_c, T_c)$. Since either \eqref{hot lower bound} or \eqref{cold lower bound} holds, we may apply Lemma \ref{expansion of positivity} once more (adjusting the constants as necessary) to find $\tilde T \ge \frac{\max\{T_h,T_c\}}{\rho^2}$ and $\nu>0$, such that 
\[
v \ge \nu \quad \text{almost everywhere in } B(x_0,2\rho)\times \left(\frac 12 \tilde T \rho^2, \tilde T \rho^2\right)
\]
if
\begin{equation}
  \label{tilde T condition}
\tilde T \rho^2 \le \left( \frac N M \right)^{m-1}(T_0-t_0).
\end{equation}
For the function $u$, this reads
\[
u \ge \frac N M \nu \quad \text{almost everywhere in } B(x_0,4\rho)
\]
for almost every
\[
t\in  \left(t_0+ \frac 12 \left(\frac M N\right)^{m-1}\tilde T \rho^2,t_0+ \left(\frac M N\right)^{m-1} \tilde T \rho^2\right).
\]
Recalling the definition of $N$ and choosing the constans $C_1=\tilde T M^{m-1}$ and $C_2=\frac M \nu$, we obtain
\[
C_2 u \ge \fint_{B(x_0,\rho)}u(x,t_0) \dx x\quad \text{almost everywhere in } B(x_0,4\rho) \times \left(t_0+\frac 12 \tau, t_0+ \tau\right ),
\] 
if \eqref{tilde T condition} holds. Here 
\[
\tau= \min\left \{ T_0-t_0, C_1 \rho^2 \left(\fint_{B(x_0,\rho)}u(x,t_0) \dx x\right)^{1-m}\right\}.
\]
If \eqref{tilde T condition} does not hold, we have 
\[
\fint_{B(x_0,\rho)}u(x,t_0)\dx x \le \left (\frac {C_1\rho^2}{T_0-t_0}\right)^{1/(m-1)}
\]
and so we conclude
\[
\fint_{B(x_0,\rho)}u(x,t_0)\dx x \le \left (\frac {C_1\rho^2}{T_0-t_0}\right)^{1/(m-1)}+ C_2 \essinf_Q u,
\]
where $Q=B(x_0,4\rho)\times (t_0+\frac 12 \tau, t_0+\tau)$, thus proving the theorem.

\end{proof}

\vspace{0.5cm}
\noindent
\small{\textsc{Pekka Lehtel\"a},}
\small{\textsc{Department of Mathematics},}
\small{\textsc{P.O. Box 11100},}
\small{\textsc{FI-00076 Aalto University},}
\small{\textsc{Finland}}\\
\footnotesize{\texttt{pekka.lehtela@aalto.fi}}


\begin{thebibliography}{000}

\bibitem{daskalopoulos}
P. Daskalopoulos and C. E. Kenig,
``Degenerate diffusions - Initial value problems and local regularity theory, 
volume 1 of EMS Tracts in Mathematics,'' European Mathematical Society (EMS), Z\"urich, 2007.

\bibitem{degiorgi}
E. De Giorgi, 
\emph{Sulla differenziabilità e l'analiticità delle estremali degli integrali multipli regolari},
Mem. Accad. Sci. Torino. Cl. Sci. Fis. Mat. Nat. (3), Vol.3 (1957), 25--43.

\bibitem{dibenedetto}
E. DiBenedetto,
``Degenerate parabolic equations,''
Universitext, Springer Verlag, New York, 1993.

\bibitem{local clustering}
E. DiBenedetto, U. Gianazza and V. Vespri,
\emph{Local clustering of the non-zero set of functions in $W^{1,1}(E)$},
Atti Acca. Naz. Lincei Cl. Sci. Fis. Mat. Natur. Rend. Lincei (9) Mat. Appl., Vol.17, (2006), 223--225.

\bibitem{harnack1}
E. DiBenedetto, U. Gianazza and V. Vespri,
\emph{Harnack estimates for quasi-linear degenerate parabolic differential equations},
Acta Math., Vol.200 (2008), 181--209.

\bibitem{harnack2}
E. DiBenedetto, U. Gianazza and V. Vespri,
``Harnack's inequality for degenerate and singular parabolic equations,''
Springer Monographs in Mathematics. Springer, New York, 2012.

\bibitem{giaquinta}
M. Giaquinta and L. Martinazzi,
``An introduction to the regularity theory for elliptic systems, harmonic maps and minimal graphs,''
Edizioni della Normale, Pisa, 2005.

\bibitem{giusti}
E. Giusti,
``Direct methods in the calculus of variations,''
World Scientific Publishing Co. Inc., River Edge, NJ, 2003.

\bibitem{ivert}
P.-A. Ivert, N. Marola and M. Masson,
\emph{Energy estimates for variational minimizers of a parabolic doubly nonlinear equation on metric measure spaces},
Ann. Acad. Sci. Fenn. Math., Vol.39 (2014), 711--719.

\bibitem{supersolutions to PME}
J. Kinnunen and P. Lindqvist,
\emph{Definition and properties of supersolutions to the porous medium equation},
J. Reine Angew. Math., Vol.618 (2008), 135--168.

\bibitem{unbounded supersolutions}
J. Kinnunen and P. Lindqvist,
\emph{Unbounded supersolutions of some quasilinear parabolic equations: a dichotomy},
Nonlinear Anal., Vol.131 (2016), 229--242.

\bibitem{KMMP}
J. Kinnunen, N. Marola, M. Miranda, Jr. and F. Paronetto,
\emph{Harnack's inequality for parabolic De Giorgi classes in metric spaces},
Adv. Differential Equations, Vol.17 (2012), 801--832.

\bibitem{kuusi}
T. Kuusi,
\emph{Harnack estimates for weak supersolutions to nonlinear degenerate parabolic equations},
Ann. Sc. Norm. Super. Pisa Cl. Sci. (5), Vol.7 (2008), 673--716.

\bibitem{shadows}
T. Kuusi, P. Lindqvist and M. Parviainen,
\emph{Shadows of Infinities},
Manuscript 2014, arXiv:1406:6309.

\bibitem{moser}
J. Moser,
\emph{A Harnack inequality for parabolic differential equations}, 
Comm. Pure Appl. Math., Vol.17 (1964), 101--134.

\bibitem{vazquez}
J. L. Vazquez,
``The Porous Medium Equation: Mathematical Theory,''
Oxford University Press, 2006.

\bibitem{nonlinear diffusion}
Z. Wu, J. Zhao, J. Yin and H. Li, 
``Nonlinear Diffusion Equations,'' 
World Scientific, Singapore, 2001.

\end{thebibliography}
\end{document}